\documentclass[a4paper]{article}
\usepackage{amssymb,amsmath,enumerate,amsthm,url}
\usepackage{mt}
\usepackage[all,cmtip]{xy}
\usepackage{graphicx}
\providecommand{\Dashv}{\mathrel{\text{\reflectbox{$\vDash$}}}}
\begin{document}
\title{Universal covers of commutative finite Morley rank groups
}

\date{Wed  7 Mar 23:41:12 CET 2018
}

\providecommand{\acl}{\operatorname{acl}}
\providecommand{\dcl}{\operatorname{dcl}}
\providecommand{\tp}{\operatorname{tp}}
\providecommand{\stp}{\operatorname{stp}}
\providecommand{\eq}{\operatorname{eq}}
\providecommand{\fin}{\operatorname{fin}}
\providecommand{\grploc}{\operatorname{grploc}}
\providecommand{\End}{\operatorname{End}}
\providecommand{\Hom}{\operatorname{Hom}}
\providecommand{\Tor}{\operatorname{Tor}}
\providecommand{\Gal}{\operatorname{Gal}}
\providecommand{\im}{\operatorname{im}}
\providecommand{\pr}{\operatorname{pr}}
\providecommand{\deg}{\operatorname{deg}}
\providecommand{\pureHull}{\operatorname{pureHull}}
\providecommand{\Ext}{\operatorname{Ext}}
\providecommand{\alg}{\operatorname{alg}}
\providecommand{\G}{\mathbb{G}}

\newcommand{\defn}{\underline}

\providecommand{\acleq}{\acl^{\eq}}
\providecommand{\dcleq}{\dcl^{\eq}}

\providecommand{\restricted}{\!\!\restriction}

\providecommand{\M}{\mathcal{M}}
\providecommand{\monst}{\mathfrak{C}}

\providecommand{\mapsonto}{\twoheadrightarrow}

\providecommand{\Qbar}{\bar{\Q}}

\renewcommand{\P}{\powerset}
\renewcommand{\Cup}{\bigcup}

\newcommand{\er}{\mathcal{O}}
\newcommand{\ek}{\er^0}

\providecommand{\pairing}[2]{ \left<{#1,#2}\right> }

\providecommand{\thmref}[1]{Theorem~\ref{#1}}
\providecommand{\lemref}[1]{Lemma~\ref{#1}}
\providecommand{\corref}[1]{Corollary~\ref{#1}}
\providecommand{\defref}[1]{Definition~\ref{#1}}
\providecommand{\propref}[1]{Proposition~\ref{#1}}
\providecommand{\factref}[1]{Fact~\ref{#1}}
\providecommand{\remref}[1]{Remark~\ref{#1}}
\providecommand{\secref}[1]{Section~\ref{#1}}
\providecommand{\ssecref}[1]{Subsection~\ref{#1}}
\providecommand{\axref}[1]{(A\ref{#1})}
\providecommand{\dedref}[1]{(D\ref{#1})}

\theoremstyle{plain}
\newtheorem{thm}{Theorem}[section]
\newtheorem{theorem}[thm]{Theorem}
\newtheorem{lemma}[thm]{Lemma}
\newtheorem{proposition}[thm]{Proposition}
\newtheorem{conjecture}[thm]{Conjecture}
\newtheorem{fact}[thm]{Fact}
\newtheorem*{fact*}{Fact}
\newtheorem{corollary}[thm]{Corollary}
\newtheorem{claim}[thm]{Claim}
\newtheorem*{claim*}{Claim}

\theoremstyle{definition}
\newtheorem{definition}[thm]{Definition}
\newtheorem*{definition*}{Definition}
\newtheorem{definitions}[thm]{Definitions}
\newtheorem*{definitions*}{Definitions}
\newtheorem{notation}[thm]{Notation}
\newtheorem*{notation*}{Notation}
\newtheorem{axioms}[thm]{Axioms}
\newtheorem{assumption}[thm]{Assumption}

\theoremstyle{remark}
\newtheorem{remark}[thm]{Remark}
\newtheorem*{remark*}{Remark}
\newtheorem{example}[thm]{Example}
\newtheorem*{example*}{Example}
\newtheorem{remarks}[thm]{Remarks}
\newtheorem*{remarks*}{Remarks}
\newtheorem{examples}[thm]{Examples}
\newtheorem*{examples*}{Examples}
\newtheorem{note}[thm]{Note}
\newtheorem*{note*}{Note}
\newtheorem{question}[thm]{Question}
\newtheorem*{question*}{Question}

\author{Martin Bays, Bradd Hart, and Anand Pillay}
%\shortauthor{M. Bays, B. Hart, and A. Pillay}

\maketitle

\begin{abstract}
  We give an algebraic description of the structure of the analytic universal
  cover of a complex abelian variety which suffices to determine the
  structure up to isomorphism. More generally, we classify the models of
  theories of ``universal covers'' of rigid divisible commutative finite
  Morley rank groups.
\end{abstract}

\section{Introduction}
\label{sec:intro}
\subsection{Characterising universal covers of abelian varieties}
\label{ssec:intro}
Let $\G = \G_m^n$ be a complex algebraic torus, or let $\G$ be a complex
abelian variety. Considering $\G(\C)$ as a complex Lie group, with
$L\G=T_0(\G(\C))$ its (abelian) Lie algebra, the exponential map provides a
surjective analytic homomorphism
    \[ \exp : L\G \twoheadrightarrow  \G(\C) .\]

Let $\er := \{ \eta \in \operatorname{End}(L\G) \;|\; \eta(\ker\exp) \subseteq  \ker\exp \} \isom
\End(\G)$ be the ring of $\C$-linear endomorphisms of $L\G$ which induce
endomorphisms of $\G(\C)$; these are precisely the algebraic endomorphisms of
$\G$. Consider $L\G$ as an $\er$-module.

% WARNING to self: remember this description of End(\G) fails for (even split)
% semiabelian

In this paper, we use model theoretic techniques and Kummer theory to give a
purely algebraic characterisation of the algebraic consequences of this
analytic picture.

\newcommand{\kerQ}{\spanofover{\ker(\exp)}{\Q}}
At first sight, $\exp$ relates $L\G$ to $\G(\C)$ in a rather particular way.
For example, if $a \in \G(\C)$ and $\exp(\alpha)=a$, then $\exp(\alpha/n)$
converges topologically to $0 \in \G(\C)$ - something which certainly needn't
hold for an arbitrary $\er$-module homomorphism. We will show however that
if we forget the topology and the analytic structure, leaving only the field
structure on $\C$ and the $\er$-module structure on $L\G$, and so work up to
field automorphisms of $\C$ and up to $\er$-module automorphisms of $L\G$,
then $\exp$ is distinguished from other $\er$-module homomorphisms only by its
interaction with the torsion subgroup $\G[\infty]$ of $\G$. More precisely, it
is described by its restriction $\exp|_{\kerQ} : \kerQ \twoheadrightarrow  \G[\infty]$ to the
divisible subgroup generated by $\ker(\exp)$: once this restriction is chosen,
there is a unique way, up to automorphisms, to extend it to $L\G$.

\begin{theorem} \label{thm:catAbelian}
    Suppose $\G$ and the action of each $\eta\in\er$ are defined over a number
    field $k_0 \leq  \C$.

    Suppose $\rho,\rho' : L\G \twoheadrightarrow  \G(\C)$ are surjective $\er$-module homomorphisms,
    $\ker\rho'=\ker\rho$, and
    $\rho'\restricted_{\spanofover{\ker\rho'}{\Q}} =
    \rho\restricted_{\spanofover{\ker\rho}{\Q}}$.

    Then there exists an $\er$-module automorphism
    $\sigma\in\Aut_\er(L\G/\ker\rho)$ and a field automorphism
    $\tau\in\Aut(\C/k_0)$ of $\C$ fixing $k_0$ such that
    the following diagram commutes, where $\tau : \G(\C) \rightarrow  \G(\C)$ is the
    abstract group automorphism induced by $\tau$.
    \[ \xymatrix{
    L\G \ar[d]^{\rho_1} \ar[r]^\sigma & L\G \ar[d]^{\rho_2} \\
    \G(\C) \ar[r]^\tau                & \G(\C) \\
    } \]
\end{theorem}

We will define an \defnstyle{$\widehat{L}$-isomorphism} to be such a pair
$(\sigma,\tau)$ of an $\er$-module isomorphism and a field isomorphism which
agree on $\G$.
So \thmref{thm:catAbelian} yields a characterisation of
$\exp : L\G \twoheadrightarrow  \G(\C)$: it is, up to $\widehat{L}$-isomorphism, the unique surjective
$\er$-homomorphism with its kernel and its restriction to the divisible
subgroup generated by that kernel.

We require here that $k_0$ is a number field in order to have Kummer theory
available. We have a corresponding result in the case that $\G$ is a split
semiabelian variety defined over a number field, but general semiabelian
varieties are problematic due to failure of Kummer theory.

We prove \thmref{thm:catAbelian} by classifying the models of the first order
theory of $\exp$ in an appropriate language $\widehat{L}$.
Our proof can be split into three stages:

\begin{enumerate}[(i)]\item Kummer theory for abelian varieties
    explains the behaviour for finite extensions of $k_0$, and
    suffices to show uniqueness of the restriction of $\exp$ to $\exp^{-1}(\G(\Qbar))$;
\item a function-field analogue of this Kummer theory allows us to extend the
    uniqueness to $\G(F)$ for $F$ an algebraically closed field of cardinality
    $\leq \aleph_1$;
\item we extend to arbitrary cardinals (in particular the continuum, which
    without assuming the continuum hypothesis is not covered by (ii)) using
    arguments involving independent systems, based on techniques involved in
    Shelah's Main Gap theorem.
\end{enumerate}

In \cite{BGHKg}, it was found that the geometric Kummer theory of (ii)
actually follows from a general model-theoretic principle, Zilber's
Indecomposability Theorem, and hence holds in the generality of rigid (see
below) commutative divisible finite Morley rank groups.

This also turns out to be a natural level of generality for (iii), and it is
in this context that we will actually work for most of this paper. We
obtain an analogue of \thmref{thm:catAbelian} in this
generality, \thmref{thm:classification} below - although since there is no
analogue of (i) in such generality we get a correspondingly weaker result.

This does allow us to remove the restrictions in \thmref{thm:catAbelian} and
still get a uniqueness result: if $\G$ is an abelian variety over a field
$k_0\leq \C$, then the exponential map $\exp : L\G \twoheadrightarrow  \G(\C)$ is, up to
$\widehat{L}$-isomorphism fixing $\exp^{-1}(\G(k_0^{\alg}))$, the unique
surjective $\End(\G)$-homomorphism with kernel $\ker\exp$ which extends
$\exp\restricted_{\exp^{-1}(\G(k_0^{\alg}))}$.
We obtain an analogous result for semiabelian varieties as part of
\ssecref{ssec:meroGrp}.

We also obtain similar results for complex tori which are not abelian
varieties, and for semiabelian varieties in positive characteristic,
generalising \cite{BZCovers}.

\subsection{Profinite covers and an outline of the paper}
For $\G = \G(\C)$ as above, or more generally for $\G$ a commutative divisible
finite Morley rank group, we associate a canonical structure $\widehat{\G}$ which we
call the ``profinite universal cover'' of $\G$, defined as the inverse limit of
copies of $\G$ with respect to the inverse system of multiplication-by-$n$
maps, $\widehat{\G} := \liminv [n] : \G \twoheadrightarrow  \G$.

In the case of $\G$ a complex semiabelian variety, this is the same
construction that appears in the definition of the \'etale fundamental group -
every finite \'etale cover of $\G$ is dominated by some $[n]$, so taking the
inverse limit with respect to all $[n]$ amounts to taking the inverse limit
with respect to all finite \'etale covers. So $\widehat{\G}$ can be identified as the
``\'etale universal cover'' of $\G$.

In general, we can see $\widehat{\G}$ as a purely algebraic substitute for an analytic
universal cover of $\G$. We will see below in Remark~\ref{rem:LGelem} one
justification for this: in an appropriate language $\widehat{L}$, if $\G$ is a Lie
group, then the Lie exponential map is an elementary submodel of the profinite
universal cover $\widehat{\G}$.

The results described in the previous subsection result from classifying the
models of the first-order theory of $\widehat{\G}$.

In \secref{sec:hat} we define the structure we wish to consider on $\widehat{\G}$,
axiomatise its first-order theory $\widehat{T}$, prove quantifier elimination, and
examine it in terms of stability theory. In \secref{sec:classification}, we
give a classification of the models of $\widehat{T}$. In \secref{sec:abelian}, we
return essentially to the context of \ssecref{ssec:intro}, specialising the
abstract model theory of earlier sections to the case of algebraic groups.
Here we also use Kummer theory to strengthen the classification (peeking
inside the prime model); the necessary Kummer theory is presented in
Appendix~\ref{sec:kummer}. Finally, in \secref{sec:examples}, we present some
further natural examples of models of $\widehat{T}$ for various $\G$, to which our
classification theorem applies.

\subsection{The literature}
We discuss the previous work on which this work builds. For $\G = \G_m$ the
multiplicative group, \thmref{thm:catAbelian} was proven in \cite{ZCovers} and
\cite{BZCovers}. It was proven for $\G$ an abelian variety in \cite{GavThesis}
under the assumption of the continuum hypothesis, i.e.\ with only the first two
of the three steps described above. A path to the full result was set out in
\cite{ZCovers2}, and for $\G$ an elliptic curve the full result was obtained
in \cite{BaysThesis}.

These previous proofs of (iii) use algebraic techniques
analogous to, but substantially more complicated and limited than, the model
theoretic techniques of the present work.
In previous work, the problem was considered one of categoricity in infinitary
logic, and correspondingly the techniques applied were those of Shelah's
theory of excellent classes, and more specifically Zilber's adaptation to
Quasiminimal Excellent (QME) classes. It was key to the developments in this
paper to instead consider the problem in terms of first-order classification
theory. Although our results do not fall literally into the context of Shelah's
classification theory for superstable theories - essentially because we are
interested in models where the kernel of $\exp$ is rather unsaturated - and
though ideas from the theory of Abstract Elementary Classes will still play a
(largely implicit) role, the argument which allows us to get (iii) in the
generality we do is an adaptation of Shelah's ``NOTOP'' argument, which
reduces the condition of excellence in the first-order case to a simpler
condition.

In fact, while the current paper was in preparation, it was found that this
same idea applies in the context of QME classes \cite{BHHKK}. For the benefit
of any readers familiar with that paper, we mention how it relates to this
paper. Our main results do not fit into the definition of QME, even if we
assume the kernel to be countable: we consider finite Morley rank groups which
are not necessarily almost strongly minimal; correspondingly, the covers are
not even almost quasiminimal. In the case discussed above of a semiabelian
variety $\G$, however, the covers structure can be seen as almost quasiminimal
- and moreover it is bi-interpretable with the quasiminimal structure induced
on the inverse image in the cover of a Kummer-generic (in the sense of
\cite{BGHKg}) curve in $\G$ which generates $\G$ as a group. So in this case,
(iii) above could be deduced from the main result of \cite{BHHKK}.

\subsection{Notation}
We use unmarked tuple notation throughout: if $A$ is a subset of a sort in
a structure, we write $x\in A$ if $x$ is a finite tuple each co-ordinate of
which is an element of $A$.

We write $a \equiv _C b$ to mean that $\tp(a/C) = \tp(b/C)$, and we sometimes write
$\sigma : A \xrightarrow{\cong}_C B$ to denote that $\sigma$ is an isomorphism which is the
identity on $C \subseteq  A\cap B$.

If $G$ is an abelian group, we write $G[n]$ for the $n$-torsion, and
we write $G[\infty]$ or $\Tor(G)$ for the torsion subgroup $\bigcup_n G[n]$.

We introduce further specialised notation in \secref{sec:hat}, after making
relevant definitions.

\subsection{Acknowledgements}

This paper grew out of the thesis of the first author, supervised by Boris
Zilber, and many of the ideas are due eventually to Zilber. The first author
was also strongly influenced by the work of, and discussions with, Misha
Gavrilovich. We would like to thank Udi Hrushovski for first suggesting the
relevance of NOTOP/PMOP. The authors would also like to thank John
Baldwin, Daniel Bertrand, Juan Diego Caycedo, Martin Hils and Jonathan Kirby
for helpful pointers and discussion at various stages of this project.

We are grateful to the anonymous referee for an earlier version of the paper,
whose suggestions and exhortations improved the paper substantially.

\section{Profinite universal covers}
\label{sec:hat}
In this section, we consider the algebra and basic model theory of our
``profinite universal covers'' of divisible commutative finite Morley rank
groups.

\subsection{$\widehat{G}$}

We begin with some elementary definitions and remarks concerning abstract
commutative groups.

If $G$ is a commutative group and $[n]$ is the multiplication-by-$n$ map, let
$\widehat{G}$ be the inverse limit $\liminv [n] : G \rightarrow  G$. Let $\rho_n : \widehat{G} \twoheadrightarrow  G$ be
the corresponding projections, so $[n]\rho_{nm} = \rho_m$. Let $\rho :=
\rho_1$. We often write elements of $\widehat{G}$ in the form $\gamma=(g_n)_n$, so
then $\rho_n(\gamma)=g_n$.

If $\theta : G \rightarrow  H$, define $\widehat{\theta} : \widehat{G} \rightarrow  \widehat{H}$ by $\widehat{\theta}((g_n)_n) =
(\theta(g_n)_n)$.

\begin{definition}
    The {\em divisible part} $G^o$ of an abelian group $G$ is the maximal
    divisible subgroup, $G^o = \bigcap_{n>0} nG$.
\end{definition}

We will mostly work in contexts in which $G^o$ is the ``connected component''
of $G$ in one sense or another, hence the notation.

Say a commutative group $G$ is {\em divisible-by-finite} if its divisible part
$G^o$ has finite index in $G$. 

We note that $\widehat{\cdot}$ is an exact functor on divisible-by-finite groups:

\begin{lemma} \label{lem:hatExact}
    Suppose $0\rightarrow A\rightarrow B\rightarrow C\rightarrow 0$ is an exact sequence of divisible-by-finite groups.
    Then $0\rightarrow \widehat{A}\rightarrow \widehat{B}\rightarrow \widehat{C}\rightarrow 0$ is exact.
\end{lemma}
\begin{proof}
    Denote the given map $B \rightarrow  C$ as $\theta$. The only difficulty is the
    surjectivity of $\widehat{\theta}:\widehat{B}\rightarrow \widehat{C}$. We may assume $A\rightarrow B$ is an inclusion.
    Factoring $\theta$ via $B/(A^o)$, we see that it suffices to prove the
    surjectivity of $\widehat{B}\rightarrow \widehat{C}$ under the assumption that $A$ is divisible or
    finite.
	%[ Remark: cf Stein decomposition\ldots here it's just a triviality! ]

    \begin{enumerate}[(a)]\item Suppose $A$ is divisible. We first show that given any $n>0$, $b\in
        B$ and $c'\in C$ such that $\theta(b)=[n]c'$, there is $b' \in B$ such
        that $[n]b' = b$ and $\theta(b')=c'$. Say $\theta(b'')=c'$; then
        $\theta([n]b'')=[n]c'=\theta(b)$, so $b-[n]b'' \in A$. Say $a'\in A$
        with $[n]a'=b-[n]b''$. Then $b' := b''+a'$ is as required.

	%   0 \rightarrow  A \rightarrow  B \rightarrow  C \rightarrow  0
	%        \;|\;    \;|\;    \;|\;
	%        v    v    v
	%   0 \rightarrow  A \rightarrow  B \rightarrow  C \rightarrow  0

	%\[ \xymatrix{
	%    0 \ar[r] & A \ar[r]\ar[d]^{[n]} & B \ar[r]\ar[d]^{[n]} & C \ar[r]\ar[d]^{[n]} & 0 \setminus
	%    0 \ar[r] & A \ar[r]             & B \ar[r]             & C \ar[r]             & 0
	%}\]

	Given $\widehat{c}$, we can therefore inductively define $b_{n!}$ such that
	$[n+1]b_{(n+1)!}=b_{n!}$ and $\theta(b_{n!}) = \rho_{n!}(\widehat{c})$. Easily, there
        is a unique $\widehat{b} \in \widehat{B}$ such that $\rho_{n!}(\widehat{b})=b_{n!}$, and it
        satisfies $\widehat{\theta}(\widehat{b})=\widehat{c}$.

    \item Suppose $A$ is finite, say $[n]A=0$. Then $\theta$ factors $[n]$ - indeed,
        let $\phi$ be the map making the left triangle in the following
        diagram commute, then note that the right triangle also commutes.
	But $\widehat{[n]}$ is surjective, hence so is $\widehat{\theta}$.

	%       [n]
	%   B \rightarrow  B
	%     \      7 \
	%    th\ phi/   \th
	%       _| /     _|
	%         C \rightarrow  C
	%            [n]

	\[ \xymatrix{
	    B \ar[rr]^{[n]} \ar[rd]_\theta &                              & B \ar[rd]^ \theta & \\
	                                   & C \ar[rr]_{[n]} \ar[ru]^\phi &                   & C
	} \]
    \end{enumerate}
\end{proof}

\subsection{$\widehat{T}$}
\label{ssec:That}
Now let $\G$ be a connected commutative finite Morley rank group, and suppose
moreover that it is divisible. Then $[n] : \G \twoheadrightarrow  \G$ has finite kernel, and
it follows that any definable subgroup $A \leq  \G$ is divisible-by-finite, and
its divisible part $A^o$ is its connected component in the
model-theoretic sense, namely the smallest definable subgroup of finite index.

Let $T := \Th(\G)$; we assume (by appropriate choice of language) that $T$ has
quantifier elimination. We also assume that the language $L$ of $T$ is countable.

Let $\widehat{T}$ be the theory of $(\widehat{\G},\G)$ in the two-sorted language $\widehat{L}$ consisting of the
maps $\rho_n$ for each $n$, the full $T$-structure on $\G$, and, for each
$\acleq(\emptyset )$-definable connected subgroup $H$ of $\G^n$, a predicate $\widehat{H}$
interpreted as the subgroup $\widehat{H} = \{ x \;|\; \Meet_n \rho_n(x) \in H \}$ of
$\widehat{\G}^n$. We will see below that $\widehat{T}$ depends only on $T$.

For quantifier elimination purposes, we actually assume (by expanding $T$ by
constants if necessary) that every $\acleq(\emptyset )$-definable connected subgroup of
$\G^n$ is $\emptyset $-definable.

We say that $T$ is {\em rigid} if for $\G$ a saturated model of $T$, every
definable connected subgroup of $\G^n$ is defined over $\acleq(\emptyset )$.
%i.e.\ if there are no infinite definable families of connected subgroups of
%any $\G^n$.
Although the results of this section do not require rigidity, our language is
chosen with it in mind.

\begin{remark}
    As in the proof of \lemref{lem:hatExact}, any definable finite group
    cover of $\G$ is dominated by some $[n]$, so is ``seen'' by $\widehat{\G}$.

    Note that divisibility is crucial for this - for example, the
    Artin-Schreier map $x \mapsto  x^p-x$ is a finite definable group cover of the
    additive group in $\operatorname{ACF}_p$ which isn't handled by our setup (c.f.\ 
    \cite{BGHKg} where this issue is discussed).

    %If we tried to relax the assumption on $\G$ to just $\omega$-stability, then we
    %would lose domination of isogenies by $[n]$. We could still try to work with the
    %inverse limit of the isogenies, but we would find that the proof of
    %\lemref{lem:hatExact}, and hence \axref{ax:proj}, could then fail. It is
    %no longer true that a connected subgroup of a divisible group is
    %divisible. TODO: check if this does actually mean the lemma fails\ldots I
    %think it does.
\end{remark}

\begin{notation}
    Suppose $(\widetilde{M},M)$ is an $\widehat{L}$-structure.
    \begin{itemize}\item 
      If $\widetilde{a} \in \widetilde{M}$ is a tuple, then we will write $a_n$ for
      $\rho_n(\widetilde{a})$, and $a$ for $\rho(\widetilde{a})$, and $\widehat{a}$ for $(a_n)_n$.
      Similarly, if $\widetilde{A} \subseteq  \widetilde{M}$, we write $\widehat{A}$ for $\Cup_n\rho_n(\widetilde{A})$.

    \item We will usually just write $\widetilde{M} \vDash  \widehat{T}$ to mean $(\widetilde{M},M) \vDash  \widehat{T}$.

    \item $\widehat{G}$ and $\widehat{H}$ will always denote the predicates corresponding to
      $\emptyset $-definable connected subgroups $G$ and $H$ of a cartesian power of
      $\G$. $\widehat{C}$ will denote a coset of some $\widehat{H}$.

    \item $\widehat{G}(\widetilde{a})$ is the definable set $\{ \widetilde{x} \;|\; (\widetilde{x},\widetilde{a}) \in \widehat{G} \}$, a coset of
      $\widehat{G}(0)$. Similarly for $G(a)$.

    \item $\ker$ is the definable set $\ker(\rho)$.

    \item $\ker^0$, the divisible part of $\ker$,
        is the $\Meet$-definable set $\Meet_n \rho_n(x) = 0$.

    \item Abusively, $\ker$ and $\ker^0$ also refer to the corresponding sets in
      cartesian powers of $\G$.

    \item If $\pr : \widetilde{M}^n \twoheadrightarrow  \widetilde{M}^m$ is a co-ordinate projection, we also write
      $\pr$ for the corresponding co-ordinate projection $M^n \twoheadrightarrow  M^m$, and
      we also write $\pr$ for the restriction of $\pr$ to a subset $\widetilde{A} \subseteq 
      \widetilde{M}^n$ or $A \subseteq  M^n$, leaving it to context to disambiguate.

    \item
      % Reviewer doesn't like this notation; consider alternatives.
      % - Considered, but I think it's fine as it is.
	$\widehat{H}_0 := \widehat{H} \cap \ker^0$, a $\Q$-subspace of the $\Q$-vector space
      $\ker^0$.
    \end{itemize}
\end{notation}

\subsection{Axiomatisation and quantifier elimination}

We now give a list of first order axioms for a structure $\widetilde{M}$ in the
language of $\widehat{T}$. We show in \propref{prop:QE} that these axioms axiomatise
$\widehat{T}$.
\begin{axioms} \mbox{}
    \begin{enumerate}[({A}1)]\item \label{ax:T}\label{ax:first} $M \vDash  T$
    \item \label{ax:grp}Let $\Gamma_+$ be the graph of the group operation on
      $\G$. Then $\widehat{\Gamma_+}$ is the graph of a commutative divisible torsion
      free group operation, which we write as ``$+$'' and work with respect to
      in the following axioms;
    \item \label{ax:diag} Let $\Delta$ be the diagonal subgroup of $\G$, i.e.\ 
    the graph of equality. Then $\widehat{\Delta}$ is the diagonal subgroup of $\widetilde{M}$.
    \item \label{ax:subgrps} Each $\widehat{H}$ is a divisible subgroup.
    \item \label{ax:coherent} $[m] \rho_{nm} = \rho_n$.
    \item \label{ax:epic} $\rho_n(\widehat{H}) = H$.
    \item \label{ax:lattice} $\widehat{G} \cap \widehat{H} = \widehat{H'^o}$ where $H':=G\cap H$.
    \item \label{ax:torsionKer} If $H\subseteq G$ and $\Tor(H)=\Tor(G)$, then
      $\widehat{H}\cap\ker=\widehat{G}\cap\ker$.
    \item \label{ax:proj}\label{ax:last}
        If a co-ordinate projection $\pr$ induces a surjection
	$\pr:G\twoheadrightarrow H$ with kernel $K$ then 
	the corresponding co-ordinate projection induces a surjection
	    $\pr:\widehat{G}\twoheadrightarrow \widehat{H}$ with kernel $\widehat{K^o}$.
    \end{enumerate}
\end{axioms}

In $\widehat{\G}$ and other models of $\widehat{T}$ which we will be considering in the
applications, $\widehat{H}$ will be the divisible part of $\rho^{-1}(H)$. In this case,
the following lemma substantially simplifies verification of the axioms.

\begin{lemma} \label{lem:homMod}
    Suppose $V$ is a divisible torsion-free abelian group,
    and $\rho : V \twoheadrightarrow  \G$ is a surjective homomorphism.
    For $H$ a connected definable subgroup of $\G^n$, let $\widehat{H}$ be the
    divisible part of $\rho^{-1}(H)$. Suppose that $\ker$ has trivial
    divisible part (i.e.\ $\widehat{0}=0$).
    Let $\rho_n(x) := \rho(x/n)$.
    Then with this structure, $V$ satisfies \axref{ax:first}-\axref{ax:last}
    if it satisfies \axref{ax:epic} and \axref{ax:proj}.
\end{lemma}
\begin{proof}
    \begin{enumerate}[({A}1)]\item %\axref{ax:T}
        Immediate.
    \item %\axref{ax:grp}
        $\widehat{\Gamma_+} = \{ (x,y,z) \;|\; \forall n. x/n + y/n - z/n \in \ker \}$,
        which, since $\ker$ has trivial divisible part, is the graph of $+$ on
        $V$.
    \item %\axref{ax:diag}
        Similar.
    \item %\axref{ax:subgrps}
        Immediate from the definition of $\widehat{H}$.
    \item %\axref{ax:coherent}
        Immediate from the definition of $\rho_n$.
    \item Assumed.
    \item %\axref{ax:lattice}
        $\widehat{G} \cap \widehat{H}$ is a divisible subgroup of $\rho^{-1}(H'^o)$ (where
        $H'=G\cap H$), so is contained in $\widehat{H'^o}$. Similarly for the
        converse inclusion.
    \item %\axref{ax:torsionKer}
        Suppose $H \subseteq  G$, $\zeta \in \widehat{G} \cap \ker$, and $\Tor(G) = \Tor(H)$.
        Then $\Q\zeta \subseteq  \rho^{-1}(H)$, so $\zeta \in \widehat{H}$ by definition of
        $\widehat{H}$.
    \item Assumed.
    \end{enumerate}
\end{proof}

\begin{lemma} \label{lem:hatMod}
    $\widehat{\G}$ satisfies the axioms \axref{ax:first}-\axref{ax:last}.
\end{lemma}
\begin{proof}
    We appeal to \lemref{lem:homMod}.
    \axref{ax:epic} and the fact that $\widehat{H}$ is the connected component of
    $\rho^{-1}(H)$ are immediate from the definitions. \axref{ax:proj} follows
    from \lemref{lem:hatExact}.
\end{proof}

\begin{proposition} \label{prop:QE}
    \axref{ax:first}-\axref{ax:last} axiomatise $\widehat{T}$, and $\widehat{T}$ has quantifier
    elimination.
\end{proposition}
\begin{proof}
    Let $\widehat{T}'$ be the theory axiomatised by \axref{ax:first}-\axref{ax:last}.

    We show that $\widehat{T}'$ is complete and admits quantifier elimination.
    Completeness and \lemref{lem:hatMod} then implies that $\widehat{T}'=\widehat{T}$.

    We first note some elementary deductions from the axioms:
    \begin{enumerate}[({D}1)]\item \label{ded:prod}
        For any $H$ and $G$, we have by \axref{ax:proj} applied to $\pr
	: H\times G \twoheadrightarrow  G$ that $\widehat{H\times G} = \widehat{H} \times \widehat{G}$.

    % (not needed in the end)
    %      \item From \axref{ax:diag} and \axref{ax:proj}(I) applied to $pr :
    %	\Delta \twoheadrightarrow  \G$, we find that $\widehat{0}$ is the zero subgroup of $\widehat{\G}$
    %	where $0$ is the zero subgroup of $\G$.

    \item \label{ded:hom}
        By \axref{ax:epic} applied to the graph of the group operation,
	the $\rho_n$ are homomorphisms.
    \item \label{ded:snake}
        In the context of \axref{ax:proj},
        if $K/K^o$ has exponent $e$, then $e\cdot(\widehat{H}\cap\ker) \subseteq  \pr(\widehat{G}\cap\ker)$.
        Indeed, this follows from \axref{ax:proj}, \axref{ax:epic}, and the
        snake lemma applied to the following diagram
        %  0  \rightarrow  ^(K^o) \rightarrow  \widehat{G}  \rightarrow  \widehat{H}  \rightarrow   0
        %           \;|\;       \;|\;      \;|\;
        %           \;|\;       \;|\;      \;|\;
        %           v       v      v
        %  0  \rightarrow     K   \rightarrow   G  \rightarrow   H  \rightarrow   0

        \[ \xymatrix{
            0 \ar[r] & \widehat{K^o} \ar[r]\ar[d]^\rho & \widehat{G} \ar[r]\ar[d]^\rho & \widehat{H} \ar[r]\ar[d]^\rho & 0 \\
            0 \ar[r] & K \ar[r]                 & G \ar[r]             & H  \ar[r]            & 0
        }. \]
    \end{enumerate}

    Now suppose we have $\omega$-saturated models $\widetilde{M},\widetilde{N}\vDash \widehat{T}'$, finite tuples
    $\widetilde{m} \equiv _{qf} \widetilde{n}$ from each with equal quantifier-free types, and a point
    $\widetilde{m}'\in \widetilde{M}$. To conclude the proof, we must find $\widetilde{n}'\in \widetilde{N}$ such that
    $(\widetilde{m},\widetilde{m}') \equiv _{qf} (\widetilde{n},\widetilde{n}')$.

    Let $\widehat{H}$ be least such that it contains $\widetilde{m}$. This exists by
    $\omega$-stability of $T$ and \axref{ax:lattice}; c.f.\ 
    \defref{defn:grploc} below.

    Let $\widehat{G}$ be least such that it contains $(\widetilde{m},\widetilde{m}')$, and let $\pr :
    (\widetilde{m},\widetilde{m}') \mapsto  \widetilde{m}$ be the co-ordinate projection.

    We work in $\widehat{T}'$; when we make a statement which is expressible as a
    sentence in $\widehat{L}$, we mean that it is a consequence of $\widehat{T}'$.

    $\pr(\widehat{G}) = \widehat{\pr(G)}$ by \axref{ax:proj}, so $\widehat{H} \subseteq  \pr(\widehat{G})$, and 
    $\pr^{-1}(\widehat{H}) = \widehat{H} \times \widehat{\G} = \widehat{H\times \G} = \widehat{\pr^{-1}(H)}$, so $\widehat{G} \subseteq 
    \pr^{-1}(\widehat{H})$ and so $\pr(\widehat{G}) \subseteq  \widehat{H}$. So $\widehat{H} = \pr(\widehat{G})$, and so
    $\pr : \widehat{G} \twoheadrightarrow  \widehat{H}$ and $\pr : G \twoheadrightarrow  H$.

    \begin{claim}
        $\pr: \widehat{G}_0(\widetilde{N}) \twoheadrightarrow  \widehat{H}_0(\widetilde{N})$
    \end{claim}
    \begin{proof}
        Work in $\widetilde{N}$.
	Let $K$ be the kernel of $\pr : G\twoheadrightarrow H$, and suppose $K/K^o$ has
        exponent $e$, so by \dedref{ded:snake},
		$\pr : \widehat{G}\cap\ker \twoheadrightarrow  e (\widehat{H}\cap\ker)$.
	So for each $k$,
		\[ \pr : k (\widehat{G}\cap\ker) \twoheadrightarrow  ke (\widehat{H}\cap\ker).\]
	But then by $\omega$-saturation of $\widetilde{N}$, 
	    \[ \pr : \widehat{G}_0 = \bigcap_k k (\widehat{G}\cap\ker) \twoheadrightarrow 
                \bigcap_k ke (\widehat{H}\cap\ker) = \widehat{H}_0 . \]
    \end{proof}

    By QE in $T$ and $\omega$-saturation, we can find $\widetilde{n}' \in \widetilde{N}$ such that
    \begin{equation}
      \tag{*}
      (\widehat{m},\widehat{m}') \equiv _{qf} (\widehat{n},\widehat{n}')
    \end{equation}
    as infinite tuples; in particular, $\rho_k(\widetilde{n},\widetilde{n}') \in G$ for all $k$, and
    so by $\omega$-saturation we find $\widetilde{n}''\in \widehat{G}(\widetilde{N})$ such that
    $\rho_k(\widetilde{n},\widetilde{n}')=\rho_k(\widetilde{n}'')$ for all $k$, and so $\widetilde{\zeta} := \widetilde{n}'' -
    (\widetilde{n},\widetilde{n}') \in \ker^0(\widetilde{N})$.
    Then $\pr \widetilde{\zeta} \in \widehat{H}_0(\widetilde{N})$, and so by the Claim there is $\widetilde{\zeta}' \in
    \widehat{G}_0(\widetilde{N})$ with $\pr \widetilde{\zeta}' = \pr \widetilde{\zeta}$, and then $\widetilde{n}''+\widetilde{\zeta}' \in \widehat{G}$
    and $\pr (\widetilde{n}''+\widetilde{\zeta}') = \widetilde{n}$.

    So we can assume $(\widetilde{n},\widetilde{n}') \in \widehat{G}$, while still satisfying (*). Now suppose
    $(\widetilde{n},\widetilde{n}')$ is contained in a proper subgroup $\widehat{G'} < \widehat{G}$. Then
    $\rho_k(\widetilde{m},\widetilde{m}')\in G'$ for each $k$, so by $\omega$-saturation, $(\widetilde{m},\widetilde{m}')
    \in \widehat{G'}+\widetilde{\zeta}$ for some $\widetilde{\zeta} \in \widehat{G}_0 \setminus \widehat{G'}_0$. So
    $\widehat{G'}_0(\widetilde{M}) < \widehat{G}_0(\widetilde{M})$, so, by \axref{ax:torsionKer},
    $\Tor(G')<\Tor(G)$. Hence by \axref{ax:epic} and \axref{ax:coherent}, for
    each $k$ there is $\widetilde{\zeta}\in \widehat{G} \setminus \widehat{G'}$ with $\rho_k(\widetilde{\zeta})=0$, and so
    by saturation $\widehat{G'}_0(\widetilde{N}) < \widehat{G}_0(\widetilde{N})$.

    Now $\pr(\widehat{G'}) = \widehat{H}$, by the same argument which showed $\pr(\widehat{G}) = \widehat{H}$,
    and so the Claim applies also to $\widehat{G}$.
    So $\pr(\widehat{G'}_0(\widetilde{N})) = \widehat{H}_0(\widetilde{N}) = \pr(\widehat{G}_0(\widetilde{N}))$.
    Hence we have a strict inclusion $\widehat{G'}_0(0) < \widehat{G}_0(0)$ in $\widetilde{N}$ for the
    fibres above $0\in \widehat{H}$. So by translating, we can find $\widetilde{n}'$ satisfying
    (*) and such that $(\widetilde{n},\widetilde{n}') \notin \widehat{G'}$. Now $\widehat{G}_0(0)$ is not covered by
    any finitely many such $\widehat{G'}_0(0)$, since they are proper $\Q$-subspaces.
    So we can avoid any finitely many such proper subgroups simultaneously,
    and so by $\omega$-saturation, we find $\widetilde{n}'$ satisfying (*) for which $\widehat{G}$
    is least such that it contains $(\widetilde{n},\widetilde{n}')$.

    It follows, using \axref{ax:diag} for formulae involving equality on the
    sort $\widehat{\G}$, that $(\widetilde{m},\widetilde{m}') \equiv _{qf} (\widetilde{n},\widetilde{n}')$ as required.
\end{proof}
\begin{remark}
    Assuming that $T$ has finite Morley rank is a much stronger assumption
    than we need for this result. Really the result is about the reduct to the
    abelian structure of $\G$ with predicates for the $\acleq(\emptyset )$-definable
    subgroups of $\G^n$; all we require is that these subgroups have divisible
    definable connected components, and the descending chain condition on
    definable subgroups. For example, $\G$ could be real semiabelian variety
    $S(\R)$ with the semialgebraic structure of its interpretation in the real
    field.
\end{remark}
\begin{remark}
    The assumption that each $\acleq(\emptyset )$-definable connected subgroup $H$ is
    actually $\emptyset $-definable in $T$ is necessary, because $H$ is $\emptyset $-definable
    in $\widehat{T}$ as the image of $\widehat{H}$, while quantifier elimination implies that
    $\G$ has only the structure of $T$.
\end{remark}

\begin{corollary} \label{cor:QE}
        Suppose $\widetilde{B} \subseteq  \widetilde{M} \vDash  \widehat{T}$, and suppose $X \subseteq  \widetilde{M}^n$ is definable over $\widetilde{B}$.
    %**(i)
        There are $H_i$, $\widetilde{b}^i \in \widetilde{B}$, $m>0$, and $\emptyset  \neq  Y_i \subseteq  H_i(b^i_m)$,
        with $i$ ranging through a finite set, and with each $Y_i$ being
        $T$-definable over $\widehat{B}$, such that
        \[ \bigcup_i \left({\widehat{H}_i(\widetilde{b}^i) \cap \rho_m^{-1}(Y_i)}\right) \subseteq  X
        \subseteq  \bigcup_i \widehat{H}_i(\widetilde{b}^i) .\]

    %**(ii)
        %If $X$ is itself a coset of the form $X = \widehat{G}+\widetilde{a}$ for some $\widehat{G}$ and any
        %$\widetilde{a} \in X$, then $X = \widehat{H}(\widetilde{b}) \cap \rho_m^{-1}(G+a_m) = \widehat{H}(\widetilde{b}) \cap (X +
        %m \ker)$ for some $\widehat{H}$, some $\widetilde{b} \in \widetilde{B}$, and some $m$.
    %*ee*
\end{corollary}
\begin{proof}
    %**(i)
        This follows from the QE, using \axref{ax:lattice} to reduce an
        intersection of cosets to a single coset,
        using \axref{ax:coherent} to reduce to a single $\rho_m$,
        and using that (by \axref{ax:epic}) $\rho_m(\widehat{H}(\widetilde{b})) \subseteq  H(b_m)$.
    %**(ii)
        %A $\Q$-vector space can not be covered by finitely many proper affine
        %subspaces, so by (i), $X$ contains a non-empty subset of the form
        %$\widehat{H}(\widetilde{b}) \cap \rho_m^{-1}(Y) \subseteq  X$ where $\widetilde{b} \in \widetilde{B}$ and $\widehat{H}(\widetilde{b}) \supseteq X$.
        %Then $\widehat{H}(\widetilde{b}) \cap \rho_m^{-1}(Y)$ is invariant under addition by
        %%$\widehat{H}(0) \cap m\ker$, so $\widehat{H}(0) \cap m\ker \subseteq  \widehat{G}$.
        %But then $\widehat{H}(\widetilde{b}) \cap \rho_m^{-1}(G+a_m) = \widehat{H}(\widetilde{b}) \cap (X+m\ker) = 
        %X + (\widehat{H}(0)\cap m\ker) = X$.
    %*ee*
\end{proof}

\begin{definition} \label{defn:grploc}
    Let $\widetilde{B} \subseteq  \widetilde{M} \vDash  \widehat{T}$, and $\widetilde{a} \in \widetilde{M}$. Then $\grploc(\widetilde{a}/\widetilde{B})$, the
    \defnstyle{group locus} of $\widetilde{a}$ over $\widetilde{B}$, is the smallest set containing
    $\widetilde{a}$ of the form $\widehat{H}(\widetilde{b})$ with $\widetilde{b} \in \widetilde{B}$.
\end{definition}

\begin{remark}
    Such a smallest set exists, by \axref{ax:lattice} and $\omega$-stability
    of $T$.

    Clearly $\grploc(\widetilde{a}/\widetilde{B})$ is definable over $\widetilde{B}$; however, it is not
    not true that $\grploc(\widetilde{a}/\widetilde{B})$ is necessarily the smallest coset of
    a $\widehat{G}$ containing $\widetilde{a}$ which is definable over $\widetilde{B}$.
    For example, suppose $\G$ is a torsion-free group, so $\rho$ is an
    isomorphism, and consider a coset $\widehat{G}+\widetilde{a}$ with $a \in \dcl^{T}(B) \setminus B$.
\end{remark}

\begin{remark} \label{rem:grplocgrp}
    Using \axref{ax:grp}, \axref{ax:lattice}, and \axref{ax:proj}, we see that
    $\widehat{H}(\widetilde{b}+\widetilde{b}')$ can be rewritten in the form $\widehat{G}(\widetilde{b},\widetilde{b}')$,
    and similarly for $\widehat{H}(\widetilde{b})+\widetilde{b}'$. So in particular,
    $\grploc(\widetilde{a}/\widetilde{B}) = \grploc(\widetilde{a}/ \left<{\widetilde{B}}\right>)$ where $\left<{\widetilde{B}}\right>$ is the
    subgroup of $\widetilde{M}$ generated by $\widetilde{B}$.
\end{remark}

\begin{lemma} \label{lem:typesKerPres}
    Let $\widetilde{B} \subseteq  \widetilde{M} \vDash  \widehat{T}$, and $\widetilde{a} \in \widetilde{M}$. Let $\widehat{C} := \grploc(\widetilde{a}/\widetilde{B})$.

    Suppose $\ker(\widetilde{M}) \subseteq  \widetilde{B}$.

    Then $p'(\widetilde{x}) := \tp(\widehat{a}/\widehat{B}) \cup \{ \widetilde{x} \in \widehat{C} \} \vDash  \tp(\widetilde{a}/\widetilde{B})$
\end{lemma}
\begin{proof}
    By the QE, we need only see that if $\widetilde{a}' \vDash  p'$ in an elementary extension,
    then for all $\widehat{H}$ and all $\widetilde{b}\in \widetilde{B}$, $\widetilde{a} \in \widehat{H}(\widetilde{b})$ iff $\widetilde{a}' \in \widehat{H}(\widetilde{b})$.

    Now $\widetilde{a} \in \widehat{H}(\widetilde{b})$ iff $\widehat{C} \leq  \widehat{H}(\widetilde{b})$, so the forward direction is clear.

    For the converse, suppose $\widetilde{a}' \in \widehat{H}(\widetilde{b})$. Then $a' \in H(b)$, hence $a \in
    H(b)$. So $(\widetilde{a},\widetilde{b}) \in \widehat{H} + \ker(\widetilde{M})$, i.e.\ $\widetilde{a} \in \widehat{H}(\widetilde{b}+\widetilde{\zeta})+\widetilde{\xi}$ for
    some $\widetilde{\zeta},\widetilde{\xi} \in \ker(\widetilde{M})$. But $\ker(\widetilde{M}) \subseteq  \widetilde{B}$,
    so by \remref{rem:grplocgrp},
    $\widehat{C} \leq  \widehat{H}(\widetilde{b}+\widetilde{\zeta})+\widetilde{\xi}$. So $\widetilde{a}' \in \widehat{H}(\widetilde{b}) \cap (\widehat{H}(\widetilde{b}+\widetilde{\zeta})+\widetilde{\xi})$; but
    this is an intersection of cosets of $\widehat{H}(0)$, so they are equal, and so
    $\widetilde{a} \in \widehat{H}(\widetilde{b})$.
\end{proof}

\begin{remark}
    It also follows from the QE that $\ker^0$ is indeed the connected
    component of the kernel in the model-theoretic sense, and more generally
    that $\widehat{H}+\ker^0$ is the connected component of $\rho^{-1}(H) = \widehat{H}+\ker$.
\end{remark}

\subsection{Lie exponential maps as models of $\widehat{T}$}
\label{ssec:lie}
Let $\G$ be a connected commutative Lie group which is also equipped with
a finite Morley rank group structure for which the model-theoretically
connected definable subgroups of $\G^n$ are topologically connected closed Lie
subgroups. This is the case for a connected commutative complex algebraic
group $\G(\C)$ with the Zariski structure, and we will discuss other examples
in \secref{sec:examples}.

Consider the Lie algebra $L\G = T_0\G$ with the Lie exponential map
$\exp : L\G \twoheadrightarrow  \G$ as an $\widehat{L}$-structure, with $\rho_m(x) := \exp(x/m)$
and $\widehat{H} := LH \leq  L\G^n$ for $H \leq  \G^n$ connected definable.

\begin{proposition} \label{prop:analMod}
    $L\G \vDash  \widehat{T}$.
\end{proposition}
\begin{proof}
    We appeal to \lemref{lem:homMod}.

    \axref{ax:epic} holds since $\exp$ is surjective for commutative Lie
    groups (since the image is a subgroup which contains a neighbourhood of
    the identity).

    So since $LH$ is divisible and $\ker\exp$ is discrete, $\widehat{H}=LH$ is the
    divisible part of $\rho^{-1}(H)$.

    Finally, \axref{ax:proj} follows from exactness of the functor $L$ for
    commutative Lie groups. To check this in the setting of \axref{ax:proj}, the
    only difficulty is the surjectivity of $LG \rightarrow  LH$, but this follows from
    the fact that the image is an $\R$-vector subspace of dimension
    $\dim(LG)-\dim(LK)=\dim(G)-\dim(K)=\dim(H)=\dim(LH)$.
\end{proof}

\begin{remark} \label{rem:LGelem}
    Note that $x \mapsto  (\exp(x/n))_n$
    is an embedding of $L\G$ into $\widehat{\G}$, which, by the QE, is elementary.
\end{remark}

\begin{remark}
    Lie theory provides a topological interpretation of the embedding of
    Remark~\ref{rem:LGelem}.

    The group $\widehat{\G}$ is easily seen to be isomorphic to the group of abstract
    group homomorphisms $\Hom(\Q,\G)$, by taking the image in $\widehat{\G}$ of $\theta
    \in \Hom(\Q,\G)$ to be $(\theta(1/n))_n$. Then by recalling that $x \mapsto  (t
    \mapsto  \exp(tx))$ is an isomorphism of $L\G$ with the group $\Hom_c(\R,\G)$ of
    1-parameter subgroups, and considering their restrictions to $\Q$, we see
    that the image in $\Hom(\Q,\G)$ of $L\G$ is precisely the subgroup
    $\Hom_c(\Q,\G)$ of continuous homomorphisms.

    By translation, $\theta \in \Hom(\Q,\G)$ is continuous iff it is continuous
    at $0$, which holds iff $\lim_{n \rightarrow  \infty} \theta(1/n) = 0 \in \G$, which
    holds iff this limit exists. So we can also identify $L\G$ as the subgroup
    of convergent elements of $\widehat{\G}$, when viewed as sequences $(a_n)_n$.
\end{remark}

\subsection{Stability theory of $\widehat{T}$}

\begin{proposition} \label{prop:forking}
    \begin{enumerate}[(i)]\item $\widehat{T}$ is superstable.
    \item If $\tp(\widetilde{a}/\widetilde{B})$ forks over $\widetilde{A}\subseteq \widetilde{B}$ then
	either $\tp(a/\widehat{B})$ forks over $\widehat{A}$
	or $\grploc(\widetilde{a}/\widetilde{B})$ is not definable over $\widetilde{A}$.
    \item $\widehat{T}$ has finite U-rank, i.e.\ $U(\widetilde{a}/\widetilde{B}) < \omega$ for any $\widetilde{a},\widetilde{B}$.
    \end{enumerate}
\end{proposition}
\begin{proof}
    \begin{enumerate}[(i)]\item
    By the QE, $\tp(\widetilde{a}/\widetilde{A})$ is determined by $\tp(\widehat{a}/\widehat{A})$ and $\grploc(\widetilde{a}/\widetilde{A})$.
    The former is determined by $\tp(a/\widehat{A})$ and $(\tp(a_k/\widehat{A}a))_k$, and since
    $[k]$ has finite kernel there are only finitely many possibilities for
    each $\tp(a_k/\widehat{A}a)$.
    $\grploc(\widetilde{a}/\widetilde{A})$ is determined by a choice of coset over $\widetilde{A}$.
    So by $\omega$-stability of $T$, if $|\widetilde{A}| = \lambda \geq  2^{|T|}$ then
    $|S(\widetilde{A})| \leq  (\lambda 2^{\aleph_0}) (|T|\lambda) = \lambda$.
    So $\widehat{T}$ is superstable.

    \item
    Suppose $\tp(\widetilde{a}/\widetilde{B})$ forks over $\widetilde{A}$;
    say $\phi(x,\widetilde{b}) \in \tp(\widetilde{a}/\widetilde{B})$ divides over $\widetilde{A}$.
    Let $\widehat{C} := \grploc(\widetilde{a}/\widetilde{B})$.
    We may assume $\phi(x,\widetilde{b}) \vDash  x\in \widehat{C}$.

    Suppose $\widehat{C}$ is defined over $\widetilde{A}$. Then also
    $\phi(x,\widetilde{b}') \vDash  x\in \widehat{C}$ for any $\widetilde{b}' \equiv _{\widetilde{A}} \widetilde{b}$.
    Now by \corref{cor:QE}, $\phi(x,\widetilde{b})$ is implied by a formula in
    $\tp(\widetilde{a}/\widetilde{B})$ of the form
	\[ x\in \widehat{C}\wedge \psi(\rho_n(x)) \]
    where $\psi(x)$ is a $T$-formula over $\widehat{B}$ implying
    $x\in \rho_n(\widehat{C})$. So since $\phi$ divides over $\widetilde{A}$,
    $\psi$ must divide over $\widehat{A}$. So $\tp(a_n/\widehat{B})$ forks over $\widehat{A}$, and since
    $a$ is algebraic over $a_n$, so does $\tp(a/\widehat{B})$.

    % Note: converse doesn't quite hold - consider case where \G is
    % torsion-free, \widetilde{a} = a \in B = \widetilde{B} is T-algebraic, and A=\emptyset.
    % (Thanks to JIMJ referee for pointing out this subtlety.)
    % Taking A=\acleq(A) isn't quite enough to deal with this problem, due to
    % the way \grploc is defined.

    \item
    Finite rankedness of $\widehat{T}$ follows from (ii), finite rankedness of
    $T$, and the fact that Morley rank bounds the length of chains of
    connected subgroups in $T$.
    \end{enumerate}
\end{proof}

We end this section with a proposition giving a stability-theoretic analysis
of $\widehat{T}$; these results are not used explicitly in the remainder of the paper,
but they inform it.
\begin{proposition}
    Let $\widetilde{\monst} \vDash  \widehat{T}$ be a monster model.

    \begin{enumerate}[(i)]\item $\ker^0$ is stably embedded, in the sense that every relatively
      definable set is relatively definable with parameters from $\ker^0$.
      Consider $\ker^0(\widetilde{\monst})$ as a structure with the
      $\emptyset $-relatively-definable sets as predicates, and let $\widehat{T}^0 :=
      \Th(\ker^0(\widetilde{\monst}))$. Then $\widehat{T}^0$ is an $\omega$-stable 1-based group of
      finite Morley rank bounded above by the Morley rank of $T$.

      In particular, $\ker^0$ has finite relative Morley rank in the sense of
      \cite{BeBoPiSSharp}.

    \item $\im(\rho)$ is stably embedded with induced structure precisely that
      of $T$.

    \item
	Every type in $\widehat{T}^{\eq}$ is analysable in $\ker^0$ and $\im(\rho)$.

    \item $\ker^0$ is orthogonal to $\im(\rho)$.

    \item A regular type in $\widehat{T}^{\eq}$ is non-orthogonal to one of
	\begin{enumerate}[(a)]\item a strongly minimal type in $T^{\eq}$;
	\item $\quot{\widehat{G}_0}{\widehat{H}_0}$ where $H \leq  G$ have
	    no intermediate connected subgroup.
	\end{enumerate}
    \item $\widehat{T}$ has weak elimination of imaginaries in $T^{\eq}$ and the sorts
      $\quot{\widetilde{\monst}^n}{\widehat{H}}$.
    \end{enumerate}
\end{proposition}
\begin{proof}
    \begin{enumerate}[(i)]\item
	%$\widehat{G}$ is divisible and torsion-free, so by the QE $\ker^0(\widetilde{M})$ is an
	%expansion by subspaces of a $\Q$-vector space.

	By the QE, the only structure on $\ker^0$ is the abelian structure given
        by the $\widehat{H}_0$. Stable embeddedness and 1-basedness follow easily. 
        (Stable embeddedness can alternatively be deduced directly from 
        stability of $\widehat{T}^0$.)

        Since $\ker^0$ is torsion-free and $\widehat{H}\cap \widehat{G} = \widehat{H'}$ where $H' = (H\cap G)^o$, the
	definable subgroups are precisely those of the form $\widehat{H}_0$. So there is
	is no infinite chain of definable subgroups of $\ker^0$, so $\widehat{T}^0$ is of
	finite Morley rank. The rank is bounded by the longest length of such a
	chain, which is bounded by the rank of $T$.

    \item This is immediate from the QE.

    \item
        Consider a strong type $q=\stp(\widetilde{a}/A)$, with $A \subseteq  \monst^{\eq}$.
        If $\widetilde{b}\in\widetilde{\monst}$ is a
	realisation of $q_1=\stp(\widetilde{a}/a\widetilde{A})$ independent from $\widetilde{a}$ over $A$,
	then since $\widehat{a} \subseteq  \acl(a)$, we have $\widetilde{a}-\widetilde{b} \in \ker^0$. So $q_1$ is
	internal to $\ker^0$, and clearly $\stp(a/A)$ is internal to
        $\im(\rho)$.

    \item It is immediate from the QE that every relatively definable subset
	of $(\ker^0)^n \times \im(\rho)^m$ is a Boolean combination of products
	of subsets of $(\ker^0)^n$ with subsets of $\im(\rho)^m$.

    \item
	By (i), the types in (b) are minimal, and $\ker^0$ is analysed in
	them. So this follows from (iii).
    \item
        For $\phi(x,y)$ an atomic formula, it is easy to see that any
        $\phi$-type over a model has canonical parameter in these sorts.
        So by the QE, any type over a model has canonical base in these sorts.
        By stability, the same holds for any type over an $\acleq$-closed set.
        Then if $\alpha = a/E \in \widetilde{\monst}^{eq}$, then $\alpha \in
        \Cb(a/\acleq(\alpha)) \subseteq  \acleq(\alpha)$.
    \end{enumerate}
\end{proof}

\section{Classification of models of $\widehat{T}$}
\label{sec:classification}

In this section, we prove the main model-theoretic result of this paper,
\thmref{thm:classification} below, which classifies the models of $\widehat{T}$.

\subsection{Outline}
\label{ssec:classificationOutline}
The classification proceeds as follows. First, recall the following coarse
version of the classification of models of $T$. By \cite[Theorem
6]{LascarFMRGrps}, $T$ is almost $\aleph_1$-categorical. It follows
(\cite[7.1]{BuechlerEssStab}) that if $M \vDash  T$ and $M_0 \prec  M$ is a copy of the
prime model, there is a finite set of mutually orthogonal strongly minimal
sets $D_i$ defined over $M_0$ such that $M$ is constructible and minimal over
$M_0B$, where $B$ is the union of arbitrary $\acl$-bases over $M_0$ for the
$D_i(M)$ (\cite[7.1.2(ii)]{BuechlerEssStab}).

We will show that this picture lifts to $\widehat{T}$. We will show that an arbitrary
model $\widetilde{M} \vDash  \widehat{T}$ is constructible and minimal over $\widetilde{M}_0B$ where $\widetilde{M}_0 =
\rho^{-1}(M_0)$, and $M_0 \prec  M$ and $B$ are as above. So models of $\widehat{T}$ are
determined up to isomorphism by a choice of model of $T$ and a choice of lift
of the prime model $M_0 \prec  M$ (which in particular involves a choice of
kernel).

In the case considered in the introduction, where $\G$ is an algebraic group
over $k_0$, we need just one strongly minimal set $D$, which we can take to be
an algebraically closed field with parameters for $k_0$. Then $M_0 \cong 
\G(k_0^{\alg})$, and for $\G(K) \vDash  T$, the basis $B$ is a transcendence base
for $K$ over $k_0^{\alg}$.

\subsection{Preliminaries}
We work in a monster model $\widetilde{\monst} \vDash  \widehat{T}$ and the corresponding monster
model $\monst=\rho(\monst) \vDash  T$.

However, we mostly want to consider only those elementary embeddings of models
of $\widehat{T}$ which preserve the kernel.

\begin{notation}
  For $\widetilde{M} \subseteq  \widetilde{N}$ models of $\widehat{T}$, we write $\widetilde{M} \prec ^* \widetilde{N}$ to mean that $\widetilde{M} \prec  \widetilde{N}$
  and $\ker(\widetilde{M}) = \ker(\widetilde{N})$. We refer to such elementary embeddings as
  {\em kernel-preserving}.
\end{notation}
\begin{remark} \label{rem:kerInvIm}
  If $\widetilde{M} \prec  \widetilde{N}$, then $\widetilde{M} \prec ^* \widetilde{N}$ iff $\widetilde{M} = \rho^{-1}(M) \subseteq  \widetilde{N}$, the inverse
  image of $M$ evaluated in $\widetilde{N}$.
\end{remark}
\begin{lemma} \label{lem:inverseStrong}
  If $\widetilde{N} \vDash  \widehat{T}$ and $M \prec  N = \rho(\widetilde{N})$, then $\widetilde{M} := \rho^{-1}(M) \prec ^* \widetilde{N}$.
\end{lemma}
\begin{proof}
  In light of Remark~\ref{rem:kerInvIm} and the quantifier elimination, it
  suffices to show that $\widetilde{M} \vDash  \widehat{T}$. For this, we check that the axioms
  \axref{ax:first}-\axref{ax:last} hold.
  These all follow straightforwardly from $M$ being an elementary
  submodel of $N$ and the kernel being preserved, except for the surjectivity
  in \axref{ax:proj} which is a little less straightforward. For that, with
  notation as in \dedref{ded:snake} of \propref{prop:QE}, note that
  $\rho_e(\widehat{H}(\widetilde{M})) = H(M) \subseteq  \pr(\rho_e(\widehat{G}(\widetilde{M})) = \rho_e(\pr(\widehat{G}(\widetilde{M}))$,
  so \begin{align*} \widehat{H}(\widetilde{M})
      &\subseteq  \pr(\widehat{G}(\widetilde{M})) + \ker(\rho_e|_{\widehat{H}(\widetilde{M})}) \\
      &=  \pr(\widehat{G}(\widetilde{M})) + e(\widehat{H}(\widetilde{M})\cap \ker) \\
      &=  \pr(\widehat{G}(\widetilde{M})) ,\end{align*}
  using that \dedref{ded:snake} holds for $\widetilde{N}$, and the kernel preservation.
\end{proof}

We make extensive use of l-isolation, a technique due to Lachlan
\cite{LachlanLIsol}.

\begin{definition}
    A type $p$ is l-isolated if for each $\phi(x,y)$ there exists $\psi(x) \in
    p$ such that $\psi$ implies the complete $\phi$-type implied by $p$, $\psi
    \vDash  p|_\phi$.
\end{definition}

Recall that $A$ is {\em atomic} over $B$ if $\tp(a/B)$ is isolated for each tuple
$a \in A$, and $A$ is {\em constructible} over $B$ if $A$ has an enumeration
$(a_i)_{i < \lambda}$ such that $\tp(a_i/Ba_{<i})$ is isolated for each
$i<\lambda$, where $Ba_{<i} = B \cup \{a_j \;|\; i<j\}$. We define {\em l-atomic} and
{\em l-constructible} by replacing isolation with l-isolation in these
definitions.

\begin{remark}
  This definition of l-isolation is easily seen to be equivalent to the
  $F^l_{\aleph_0}$-isolation of \cite[Definition~IV.2.3]{ShCT}.

  Clearly any isolated type is l-isolated, so atomicity implies l-atomicity
  and constructibility implies l-constructibility.

  It is easy to see that, just as for constructibility and atomicity in their
  usual senses, l-constructibility implies l-atomicity
  (\cite[Theorem~IV.3.2]{ShCT}), and the converse holds for countable
  sets (\cite[Lemma~IV.3.16]{ShCT}).
\end{remark}

\begin{lemma} \label{lem:lPrim}
    \begin{enumerate}[(a)]\item Work in a monster model $\monst'$ of a complete stable theory.
	\begin{enumerate}[(i)]\item l-constructible models exist over arbitrary sets: for $A \subseteq 
          \monst'^{\eq}$, there exists $M \prec  \monst'$ such that $A \subseteq  M^{\eq}$
          and $M^{\eq}$ is l-constructible over $A$.
        \item If $M \prec  \monst'$, and $\phi$ is a formula over $M$ such that
          $\phi(M) \subseteq  A \subseteq  \monst'^{\eq}$
	  and $\dcleq(A) \cap \dcleq(\phi(\monst')) \subseteq  M^{\eq}$,
		and if $b$ is l-isolated over $A$ and $\vDash  \phi(b)$,
		then $b\in \phi(M)$.
	\end{enumerate}
    \item If $\widetilde{M} \vDash  \widehat{T}$ and $\rho(\widetilde{M}) =: M \prec  N \vDash  T$, and $\widetilde{N} \vDash  \widehat{T}$ is l-atomic over $A:=\widetilde{M}\cup N$,
	then $\widetilde{N} \prec ^* \widetilde{M}$ and $\rho(\widetilde{N}) = N$, and so $\widetilde{N}$ is minimal over $A$.
    \end{enumerate}
\end{lemma}
\begin{proof}
    \begin{enumerate}[(a)]\item \mbox{ }
	\begin{enumerate}[(i)]\item \cite[IV.2.18(4), IV.3.1(5)]{ShCT}
	\item If $b\notin \phi(M)$, then by l-isolation, there is a formula
		$\psi \in \tp(b/A)$ such that
		    \[ \psi(x) \vDash  \phi(x)\wedge x\notin \phi(M) . \]
                By stable embeddedness of $\phi$, we may take $\psi$ to be
                defined over
		  $\dcleq(A) \cap \dcleq(\phi(\monst')) \subseteq  M^{\eq}$ .
		But then $\psi$ is realised in $M$, which is a contradiction.
	\end{enumerate}
    \item This follows from (a)(ii) and the QE.
      Indeed, if $\beta \in \dcleq(\widetilde{\zeta})$ with $\widetilde{\zeta}$ a tuple from
      $\ker(\widetilde{\monst})$, then since $\rho_n(\widetilde{\zeta}) \in \Tor \subseteq  M$, the QE implies
      that $\tp(\widetilde{\zeta}/A)$ is determined by $\tp(\widetilde{\zeta}/\widetilde{M})$. So if $\beta \in
      \dcleq(A)$, then already $\beta \in \widetilde{M}^{\eq}$. So by (a)(ii), $\ker(\widetilde{N})
      = \ker(\widetilde{M})$.

      Similarly, if $\beta \in \dcleq(b)$ with $b$ a tuple from $\monst$, 
      then the QE implies that $\tp(b/A)$ is determined by $\tp(b/N)$. Let $\widehat{N}
      \vDash  \widehat{T}$ be the profinite universal cover, embedded in $\widetilde{\monst}$ over $N$.
      Then $\tp(b/A)$ is determined by $\tp(b/\widehat{N})$, again, if $\beta \in
      \dcleq(A)$, then already $\beta \in \widehat{N}^{\eq}$. So by (a)(ii), we have
      $\rho(\widetilde{N}) = \rho(\widehat{N}) = N$.
      
      The claimed minimality is then clear, since $\rho$ is a homomorphism.
    \end{enumerate}
\end{proof}

We also use the existence of l-constructible models to obtain independent
amalgamation in the (abstract elementary) class $(\operatorname{Mod}(\widehat{T}),\prec ^*)$ of models
of $\widehat{T}$ with kernel-preserving embeddings.
\begin{lemma} \label{lem:amalg}
    Suppose $\widetilde{M}_i$, $i=0,1,2$, are elementary submodels of $\widetilde{\monst}$,
    $\widetilde{M}_0 \prec ^* \widetilde{M}_i$, and $\widetilde{M}_1 \ind _{\widetilde{M}_0} \widetilde{M}_2$. Let $\widetilde{M}_3$ be an
    l-atomic model over $\widetilde{M}_1 \cup \widetilde{M}_2$. Then $\widetilde{M}_i \prec ^* \widetilde{M}_3$.
\end{lemma}
\begin{proof}
    Suppose $\widetilde{\zeta} \in \ker(\widetilde{M}_3) \setminus \ker(\widetilde{M}_0)$.
    By l-atomicity, say $\phi(x,\widetilde{a}_1,\widetilde{a}_2) \in \tp(\widetilde{\zeta}/\widetilde{M}_1\cup \widetilde{M}_2)$ with
    $a_i \in \widetilde{M}_i$ and
        \[ \phi(x,\widetilde{a}_1,\widetilde{a}_2) \vdash  x \notin \ker(\widetilde{M}_0)\wedge x \in \ker .\]
    By \corref{cor:QE}, we may assume $\phi(x,\widetilde{a}_1,\widetilde{a}_2)$ is of the form
    $x \in \widehat{H}(\widetilde{a}_1,\widetilde{a}_2)\wedge \rho_n(x) = \zeta_n$, where $\zeta_n =
    \rho_n(\widetilde{\zeta}) \in M_0$.

    By the independence, $\tp(\widetilde{a}_1/\widetilde{M}_2)$ is finitely satisfiable in $\widetilde{M}_0$,
    so say $\widetilde{a}_1' \in \widetilde{M}_0$ and $\widetilde{M}_2 \vDash  \exists x \in \ker. \phi(x,\widetilde{a}_1',\widetilde{a}_2)$,
    witnessed say by $\widetilde{\zeta}' \in \ker(\widetilde{M}_2) = \ker(\widetilde{M}_0)$.
    
    Then $\tp(\widetilde{\zeta}-\widetilde{\zeta}'/\widetilde{M}_1) \ni (x \in \widehat{H}(\widetilde{a}_1-\widetilde{a}_1',0)\wedge \rho_n(x) = 0)$,
    so say $\widetilde{\zeta}'' \in \widetilde{M}_1$ also satisfies this.
    Then $\widetilde{\zeta}' + \widetilde{\zeta}'' \in \widehat{H}(\widetilde{a}_1'+\widetilde{a}_1-\widetilde{a}_1',\widetilde{a}_2+0) = \widehat{H}(\widetilde{a}_1,\widetilde{a}_2)$
    and $\rho_n(\widetilde{\zeta}'+\widetilde{\zeta}'') = \zeta_n + 0 = \zeta_n$; but $\widetilde{\zeta}' +
    \widetilde{\zeta}'' \in \ker(\widetilde{M}_1) = \ker(\widetilde{M}_0)$, contradicting the choice of $\phi$.
\end{proof}

\begin{lemma} \label{lem:atomCons}
    Suppose $\widetilde{M} \vDash  \widehat{T}$ and $A \subseteq  \widetilde{M}^{\eq}$ with $\ker(\widetilde{M}) \subseteq  A$, suppose $M$ is
    countable, and suppose $\widetilde{M}$ is atomic over $A$. Then $\widetilde{M}$ is
    constructible over $A$.
\end{lemma}
\begin{proof}
    Take an arbitrary section $S \subseteq  \widetilde{M}$ of $\rho : \widetilde{M} \rightarrow  M$. Then $S$ is
    countable and atomic, and hence constructible, over $A$, and $\widetilde{M} = S+\ker$
    is clearly constructible over $S\cup A \supseteq S\cup\ker$.
\end{proof}

\begin{lemma} \label{lem:isolLIsol}
    Suppose $\widetilde{B} \subseteq  \widetilde{M} \vDash  \widehat{T}$ and $\widetilde{a} \in \widetilde{M}$, and each $\tp(a_m/\widehat{B})$ is
    isolated. Then $\tp(\widetilde{a}/\widetilde{B})$ is l-isolated.
\end{lemma}
\begin{proof}
    By the QE, it suffices to see that $\tp_{\phi}(\widetilde{a}/\widetilde{B})$ is isolated for an
    atomic formula $\phi(x,y)$. For $\phi$ of the form
    $\psi(\rho_m(x),\rho_m(y))$, this follows from $\tp(a_m/\widehat{B})$ being
    isolated. For $\phi$ of the form $(x,y) \in \widehat{H}$, it follows from the fact
    that for $\widetilde{b} \in \widetilde{B}$, $(\widetilde{a},\widetilde{b}) \in \widehat{H} \Leftrightarrow  \grploc(\widetilde{a}/\widetilde{B}) \subseteq  \widehat{H}(\widetilde{b})$.
\end{proof}

\subsection{$\omega$-stability over models}
From now on, in order to prove the subsequent lemma, we make the following
additional assumption.

\begin{assumption}
    $T$ is {\em rigid} - for $\G$ a saturated model of $T$, every connected
    definable subgroup $H$ of $\G^n$ is defined over $\acleq(\emptyset )$ - and hence,
    by our previous assumptions in \ssecref{ssec:That}, is actually defined
    over $\emptyset $. So $\widehat{L}$ has a predicate $\widehat{H}$ corresponding to $H$.
\end{assumption}

We now apply the ``Kummer theory over models'' of \cite{BGHKg} to obtain
atomicity of ``finitely generated'' extensions of models.

\begin{lemma} \label{lem:wstab}
    Suppose $\widetilde{M} \prec  \widetilde{\monst}$ and $b \in \monst$,
    and let $M(b)$ be a prime model over $Mb$.
    Suppose $\widetilde{M}(b)$ is a model such that $\widetilde{M} \prec ^* \widetilde{M}(b) \prec  \widetilde{\monst}$ and
    $\rho(\widetilde{M}(b)) = M(b)$.
    Then $\widetilde{M}(b)$ is atomic over $\widetilde{M}b$. If $M$ is countable, $\widetilde{M}(b)$ is
    constructible over $\widetilde{M}b$.

    Furthermore, such an $\widetilde{M}(b)$ exists.
\end{lemma}
\begin{proof}
    We first show the atomicity.
    Let $\widetilde{c} \in \widetilde{M}(b)$; we must show that $\tp(\widetilde{c}/\widetilde{M}b)$ is isolated.
    Let $\widehat{H}+\widetilde{d} = \grploc(\widetilde{c}/\widetilde{M})$. Since $\widetilde{M}$ is a model, we may assume $\widetilde{d} \in
    \widetilde{M}$. So by replacing $\widetilde{c}$ with $\widetilde{c}-\widetilde{d}$, we may assume $\widetilde{d}=0$.

    Since $\widetilde{M}$ contains $\ker(\widetilde{M}(b))$ and $T$ is rigid, $c$ is {\em free} in $H$
    over $M$, i.e.\ in no proper coset defined over $M$. By \cite[6.4]{BGHKg},
    for some $n$, writing $\widehat{x}$ for the long tuple of variables $(x_i)_{i>0}$,
    \[ \tp(c_n/M)(x_n) \cup \{ x_i \in H \;|\; i>0 \} \vDash  \tp(\widehat{c}/M)(\widehat{x}). \]

    Now by $\omega$-stability of $T$, $\tp(b/Mc)$ has finite multiplicity,
    i.e.\ finitely many extensions to $\acleq(Mc) \supseteq \widehat{c}$. Hence $\tp(\widehat{c}/M) \cup
    \tp(c/Mb)$ has only finitely many extensions to Mb. So again, for some
    $n$,
    \[ \tp(c_n/Mb)(x_n) \cup \{ x_i \in H \;|\; i>0 \} \vDash  \tp(\widehat{c}/Mb)(\widehat{x}). \]

    So by \lemref{lem:typesKerPres},
	\[ \tp(c_n/Mb)(\rho_n(\widetilde{x})) \cup \{\widetilde{x} \in \widehat{H}\} \vDash  \tp(\widetilde{c}/\widetilde{M}b)(\widetilde{x}). \]
    But $c_n \in M(b)$, so $\tp(c_n/Mb)$ is isolated, so $\tp(\widetilde{c}/\widetilde{M}b)$ is
    isolated.

    This proves atomicity. Constructibility assuming countability of $M$
    follows by \lemref{lem:atomCons}.

    It remains to show existence. By \lemref{lem:lPrim}(a)(i), there exists a
    model $\widetilde{M}(b)$ which is l-constructible over $\widetilde{M} \cup M(b)$, and by
    \lemref{lem:lPrim}(b) the kernel is preserved and $\rho(\widetilde{M}(b)) = M(b)$.

    % I considered doing the following, but decided against it, as it would
    % really just be replicating what's in the KG paper. The proof there is
    % good and clear as it stands.
    %
    % Since we already have \widehat{\G}, we could give a direct proof without
    % citing BGHKg easily enough (though it's the same proof really), talking
    % about action of Aut(\G/Tor) on \widehat{\G}; we get an open subgroup of
    % \widehat{H}\cap\ker as a finite sum by ZIT, so must have been open. Might be
    % worth explaining this.
    % (key point: by omega stability, within an absolute galois group, the
    % fixator of a finite tuple is of finite index)
\end{proof}

\begin{remark}
    Note that $\widetilde{M}(b)$ will {\bf not} be constructible over $\widetilde{M} \cup M(b)$: indeed,
    if $\widetilde{a}\in \widetilde{M}(b) \setminus \widetilde{M}$, then each $a_n$ is in $M(b) \setminus M$, so easily
    $\tp(\widetilde{a}/\widetilde{M}\cup M(b))$ is not isolated.
\end{remark}

\begin{remark}
    If we don't assume rigidity, there could be subgroups definable over
    $M(b)$ which aren't definable over $M$, which could cause a failure of
    atomicity.
\end{remark}

\begin{remark}
    \lemref{lem:wstab} implies that we have $\omega$-stability over models in
    the (abstract elementary) class $(\operatorname{Mod}(\widehat{T}),\prec ^*)$,
    in the sense that if $\widetilde{M} \vDash  \widehat{T}$ is countable, then there are only
    countably many types over $\widetilde{M}$ realised in kernel-preserving extensions of
    $\widetilde{M}$. Indeed, by \lemref{lem:wstab} any such type is isolated over $\widetilde{M}b$
    for some $b$, and by $\omega$-stability of $T$ there are only countably
    many possible types $\tp(b/\widetilde{M}) \Dashv  \tp(b/M)$. We will see in
    \remref{rem:homog} below that the Galois type of $b$ over $\widetilde{M}$ is
    determined by $\tp(b/\widetilde{M})$, which means that we have $\omega$-stability
    over models in the sense of the abstract elementary class.
\end{remark}

\subsection{Independent systems}
Countability of $M$ was crucial to get constructibility in \lemref{lem:wstab}.
For constructibility of extensions in higher cardinals, we require
constructibility over independent systems of models.
\cite[XII]{ShCT} and \cite{HartOTOP} are the sources for the techniques used
here.

In this subsection, we develop what we need of the general theory of
independent systems. We work in a monster model $\monst'$ of an arbitrary
stable theory $T'$.

\begin{definition}
    If $I$ is a downward-closed set of sets, an \defnstyle{$I$-system} in
    $\monst'$ is a collection $(M_s \;|\; s \in I)$ of elementary submodels
    $\monst'$ such that for $s \subseteq  t$, $M_s$ is an elementary submodel of
    $M_t$. For $J \subseteq  I$, define $M_J := \Cup_{s\in J} M_s \subseteq  \monst'$.

    Define ${{<}s} := \P^-(s) := \P(s) \setminus \{s\}$,
    and ${{\not\geq }s} := I \setminus \{t \;|\; t \supseteq s\}$.

    The system is {\em constructible} if $M_s$ is constructible over $M_{<s}$ for
    all $s\in I$ with $|s|>1$. Similarly for {\em atomic}, and for
    {\em l-constructible} and {\em l-atomic}.

    The system is {\em independent} (or {\em stable}) if $M_s \ind _{M_{<s}} M_{\not\geq s}$
    for all $s\in I$.

    $I$ is {\em Noetherian} if each $s \in I$ is finite.

    An {\em enumeration} of $I$ is a sequence $(s_i)_{i\in\lambda}$ such that
    $I=\{s_i \;|\; i\in\lambda\}$ and $s_i \subseteq  s_j \Rightarrow  i \leq  j$. We then write
    $s_{<i}$ for $\{ s_j \;|\; j<i \}$.

    We define $|n| := \{0,\ldots,n-1\}$.
\end{definition}

Note that if $(s_i)_{i \in \lambda}$ is an enumeration of an independent
$I$-system, then we have $M_{s_i} \ind _{M_{<s_i}} M_{s_{<i}}$ for all $i$.
That the converse holds is given by the following Fact, which is
\cite[Lemma~XII.2.3(1)]{ShCT}.
\begin{fact} \label{fact:indieViaEnum}
  Let $(M_s)_s$ be an $I$-system, let $(s_i)_{i\in\lambda}$ be an enumeration,
  and suppose $M_{s_i} \ind _{M_{<s_i}} M_{s_{<i}}$ holds for all $i$. Then the
  system is independent.
\end{fact}

\begin{definition}
    Let $M$ be a (possibly multi-sorted) structure. If $A \subseteq  B \subseteq  M$, we say
    $A$ is {\em Tarski-Vaught} in $B$, $A \subseteq _{TV} B$, if every formula over $A$
    which is realised in $B$ is realised in $A$.
\end{definition}
\begin{lemma} \label{lem:TVAndLIsolation}
    Suppose $C \subseteq _{TV} B \subseteq  M$.
    \begin{enumerate}[(i)]\item
        If a type $\tp(a/C)$ is l-isolated, then $\tp(a/C) \vDash  \tp(a/B)$,
        and $Ca \subseteq _{TV} Ba$.
    \item
	If $A \subseteq  M$ is constructible over $C$ then $A$ is constructible over $B$.
    \item
	If $A \subseteq  M$ is l-atomic over $C$, then $A \ind _C B$.
    \end{enumerate}
\end{lemma}
\begin{proof}
    \begin{enumerate}[(i)]\item
      Given $\phi(x,y)$, say $\psi(x) \in \tp(a/C)$ isolates $\tp_\phi(a/C)$.
      Then for $b \in B$, $\phi(x,b) \in \tp_\phi(a/B)$ iff $\psi(x) \vDash 
      \phi(x,b)$;
      indeed, else
      \[ \vDash  (\exists x. \psi(x)\wedge \phi(x,b))
       \wedge (\exists x. \psi(x)\wedge \neg\phi(x,b)) ;\]
      but then the same holds for some $c \in C$, contradicting the isolation.

      So $\tp(a/C) \vDash  \tp(a/B)$.
      Also $Ca \subseteq _{TV} Ba$, since if $\vDash  \phi(a,b)$ then $\vDash  \forall x.
      \psi(x) \rightarrow  \phi(x,b)$, hence this holds for some $c \in C$, and hence $\vDash 
      \phi(a,c)$.
    \item
        This follows from (i) by a transfinite induction.
    \item
        The extension of $\tp(A/C)$ to $B$ is unique by (i), so must be the
        non-forking extension.
    \end{enumerate}
\end{proof}

\begin{lemma} \label{lem:TVAndCoheirs}
    Suppose $M$ is a model, and $A \ind _M B$. Then $MA \subseteq _{TV} MAB$.
\end{lemma}
\begin{proof}
    By the coheir property of non-forking over models in stable theories,
    $\tp(B/MA)$ is finitely satisfiable in $M$.
\end{proof}

The following is \cite[Lemma~XII.2.3(2)]{ShCT}, to which we refer for the proof.
\begin{fact}[TV Lemma] \label{fact:TVLemma}
    If $(M_s)_s$ is an independent $I$-system in a stable theory, if $J \subseteq  I$,
    and if $\forall s\in I. (s \subseteq  \Cup J \Rightarrow  s \in J)$, then $M_J \subseteq _{TV} M_I$.
\end{fact}

\begin{lemma} \label{lem:sysPrimOverBasis}
    Let $(M_s)_s$ be a constructible Noetherian independent $I$-system.
    Suppose that for each $p \in \Cup I$, $B_p$ is a subset of $M_{\{p\}}$ for
    which $M_{\{p\}}$ is constructible over $M_{\emptyset }B_p$.

    Then $M_I$ is constructible over $M_{\emptyset }\cup\Cup_{p\in \Cup I}B_p =: A$.
\end{lemma}
\begin{proof}
  Let $(s_i)_{i<\lambda}$ be an enumeration of $I$. It suffices to show that
  each $M_{s_i}$ is constructible over $AM_{s_{<i}}$, as it then follows
  by induction on $i\leq \lambda$ that $M_{s_{<i}}$ is constructible over
  $A$.

  If $s_i = \emptyset $, the constructibility is immediate.
  If $s_i = \{p\}$, we have $B_p \ind _{M_{\emptyset }} (M_{s_{<i}}\cup(A \setminus B_p))$.
  So by \lemref{lem:TVAndCoheirs},
  $M_{\emptyset }B_p \subseteq _{TV} AM_{s_{<i}}$.
  The desired constructibility then follows from
  \lemref{lem:TVAndLIsolation}(ii).

  If $|s_i|>1$, then $M_{s_i}$ is constructible over $M_{<s_i}$; but $M_{<s_i}
  \subseteq _{TV} M_{s_{<i}\cup \{\{p\} \;|\; p \in \Cup I\}}$ by the TV Lemma, so in
  particular $M_{<s_i} \subseteq _{TV} AM_{s_{<i}}$. Again,
  \lemref{lem:TVAndLIsolation}(ii) yields the desired constructibility.
\end{proof}
    
\begin{lemma} \label{lem:primIndieSys}
    Suppose $\Cup I$ is finite.

    For an l-atomic $I$-system to be independent,
    it suffices that for each $p \in \Cup I$,
	\[ M_{\{p\}} \ind _{M_\emptyset } M_{\not\geq \{p\}}. \]
\end{lemma}
\begin{proof}
    Suppose inductively that for any downward closed proper subset $J$ of $I$, the
    restriction of the $I$-system to a $J$-system is independent.

    So it suffices to show that for $s \in I$ maximal, $M_s \ind _{M_{<s}} M_{I \setminus \{s\}}$.

    If $|s| = 1$, this holds by assumption.

    If $|s|>1$, then if $t\subseteq s$ and $t \in I \setminus \{s\}$ then $t \in {{<}s}$, so by the TV Lemma
    applied to the restricted independent $(I \setminus \{s\})$-system,
	\[ M_{<s} \subseteq _{TV} M_{I \setminus \{s\}}. \]

    But $M_s$ is l-atomic over $M_{<s}$, so we conclude the independence by
    \lemref{lem:TVAndLIsolation}(iii).
\end{proof}

\subsection{Atomicity over independent systems in $\widehat{T}$}
Now we return to considering $\widehat{T}$ and $T$.

Let $M \vDash  T$, and let $M_0 \prec  M$ be a copy of the prime model, and let $D_i$
and $B_i$ be as in \ssecref{ssec:classificationOutline}.
Let $B:=\Cup_i B_i$, and let $\P^{\fin}(B)$ be the set of finite subsets of
$B$. Let $M_{\emptyset }=M_0$, and for $s\in \P^{\fin}(B)$ inductively let $M_s \prec  M$
be prime over $M_{<s} \cup s$.

\begin{lemma} \label{lem:baseSys}
    $(M_s)_{s \in \P^{\fin}(B)}$ is a constructible independent $\P^{\fin}(B)$-system, and $\Cup_s M_s = M$.
\end{lemma}
\begin{proof}
    $\Cup_s M_s$ is an elementary submodel of $M$ which contains $M_0B$, and
    $M$ is minimal over $M_0B$, so $\Cup_s M_s = M$.

    The system is constructible by construction, prime models being
    constructible in $\omega$-stable theories.
    For independence, by finite character of forking and
    \lemref{lem:primIndieSys} it suffices to see that $M_{\{b\}} \ind _{M_0} M_s$
    when $b \notin s \in \P^{\fin}(B)$.

    We may assume inductively that the restriction of the system to $s$ is
    independent. So by \lemref{lem:sysPrimOverBasis}, $M_s$ is constructible over $M_0s$.

    Now $b \notin M_s$ since (by orthogonality of the $D_i$) $\tp(b/M_0s)$ is not
    algebraic and hence not isolated.

    So $bM_0 \ind _{M_0} M_s$. So by \lemref{lem:TVAndCoheirs}, $bM_0 \subseteq _{TV} bM_s$.
    Since $M_{\{b\}}$ is constructible and hence atomic over $bM_0$, it follows by
    \lemref{lem:TVAndLIsolation}(iii) that $M_{\{b\}} \ind _{bM_0} bM_s$, and in
    particular $M_{\{b\}} \ind _{bM_0} M_s$.
    So by transitivity, $M_{\{b\}} \ind _{M_0} M_s$.
\end{proof}

\begin{definition}
    An $I$-$\sim$-system is an $I$-system $(\widetilde{M}_s)_s$ in $\widetilde{\monst} \vDash  \widehat{T}$ such that
    \begin{itemize}\item setting $M_s := \rho(\widetilde{M}_s) \vDash  T$,
	$(M_s)_s$ is an independent atomic $I$-system in $T$;
    \item $\widetilde{M}_s \prec ^* \widetilde{M}_t$ when $s \subseteq  t$.
    \end{itemize}
\end{definition}

The definition assumes only independence in $T$, but independence in $\widehat{T}$
follows:
\begin{lemma} \label{lem:tildeSysIndie}
    An $I$-$\sim$-system $(\widetilde{M}_s)_s$ is an independent $I$-system.
\end{lemma}
\begin{proof}
    Let $(s_i)_{i\in\lambda}$ be an enumeration of $I$. By
    \factref{fact:indieViaEnum}, it suffices to show that given
    $i\in\lambda$, we have $\widetilde{M}_{s_i} \ind _{\widetilde{M}_{<s_i}} \widetilde{M}_{s_{<i}}$, where we
    may assume inductively that the restriction of $(\widetilde{M}_s)_s$ to $s_{<i}$ is an
    independent system.

    By \propref{prop:forking}(ii) and the independence of $(M_s)_s$, it suffices to
    show that for $\widetilde{a} \in \widetilde{M}_{s_i}$, we have $C := \grploc(\widetilde{a}/\widetilde{M}_{s_{<i}})$ is
    defined over $\widetilde{M}_{<s_i}$.
    Say $C = \widehat{H}(\widetilde{b}')$ with $\widetilde{b}' \in \widetilde{M}_{s_{<i}}$.

    Now $aM_{<s_i} \subseteq _{TV} aM_{s_{<i}}$, by the TV Lemma
    (\factref{fact:TVLemma}) and \lemref{lem:TVAndLIsolation}(i) if $|s_i|>1$,
    and by \lemref{lem:TVAndCoheirs} if $|s_i| = 1$.
    
    So $H(b') = H(0) + a = H(b)$ for some $b \in M_{<s_i}$. So say
    $\widetilde{b} \in \widetilde{M}_{<s_i}$ and $\rho(\widetilde{b})=b$; then $\widetilde{a} + \zeta \in \widehat{H}(\widetilde{b})$ for some
    $\zeta \in \ker(\widetilde{M}_{s_i})=\ker( \widetilde{M}_{\emptyset } )$. So $\widehat{H}(\widetilde{b}') = \widehat{H}(\widetilde{b}) - \zeta$,
    which (by \remref{rem:grplocgrp}) is defined over $\widetilde{M}_{<s_i}$.
\end{proof}

\begin{proposition} \label{prop:indieAtomic}
    Let $(\widetilde{M}_s)_s$ be an $I$-$\sim$-system with $I$ Noetherian.
    Then the system is atomic.
    If also each $M_s$ is countable, then the system is constructible.
\end{proposition}
\begin{proof}
    It suffices to show this for $I=\P(|n|)$, $n>1$, where recall
    $|n\vDash \{0,\ldots,n-1\}$. Indeed, by Noetherianity,
    the system below any $s \in I$ is of this form. We inductively assume the
    result for $1<n'<n$.

    We show that $\widetilde{M}_{|n|}$ is atomic over $\widetilde{M}_{<|n|}$. Constructibility
    assuming countability then follows by \lemref{lem:atomCons}.

    \begin{claim}
	$(\widetilde{M}_s)_s$ extends to a $\P(|n+1|)$-$\sim$-system such that $\widetilde{M}_{|n|}$ is isomorphic
        over $\widetilde{M}_{|n-1|}$ to $\widetilde{M}_{ |n-1| \cup \{n\} }$, by an isomorphism
        $\sigma$ such that moreover $\sigma(\widetilde{M}_s) = \widetilde{M}_{(s \setminus \{n-1\}) \cup
        \{n\}}$ for $s \subseteq  |n|$.
    \end{claim}
    \begin{proof}
	Let $t := |n-1| \cup \{n\}$.

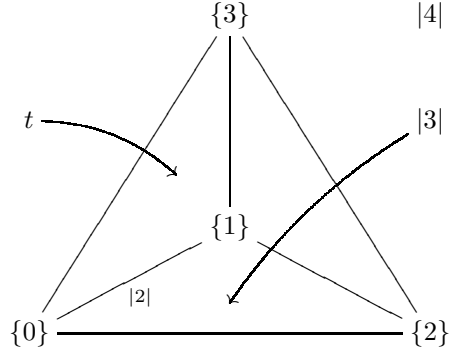
    \[
    \xymatrix{
	                         &  & \{3\}\ar@{-}[llddd] \ar@{-}[dd] \ar@{-}[rrddd] &  & |4| \\
	 t \ar@/^/[rr]+/ld 1cm/  &  &                                                &  & |3| \ar@/_/[lld]+/d 1cm/ \\
	                         &  & \{1\}\ar@{-}[lld]^{|2|} \ar@{-}[rrd]           &  & \\
	\{0\} \ar@{-}[rrrr]      &  &                                                &  & \{2\}
    } \]

        Let $\widetilde{M}_t$ be a realisation of $\tp(\widetilde{M}_{|n|}/\widetilde{M}_{|n-1|})$ independent
        from $\widetilde{M}_{|n|}$, and let $\sigma : \widetilde{M}_{|n|} \xrightarrow{\cong}_{\widetilde{M}_{|n-1|}} \widetilde{M}_t$
        be an isomorphism witnessing the equality of types. Let $\widetilde{\M}$ be an
        l-atomic model over $\widetilde{M}_{|n|} \cup \widetilde{M}_t$. By \lemref{lem:amalg},
        $\ker(\widetilde{\M}) = \ker(M_\emptyset )$.

	We define an enumeration $s_i$ of $\P(|n+1|)$, and recursively define
        $\widetilde{M}_{s_i} \prec ^* \widetilde{\M}$ such that
	    \[ M_{s_i} \ind _{M_{<s_i}} M_{s_{<i}} \]
	and $M_{s_i}$ is atomic over $M_{<s_i}$.
        By \factref{fact:indieViaEnum}, this will yield a
        $\P(|n+1|)$-$\sim$-system.

	Begin with an enumeration of $\P(|n|)$; the corresponding
        $\widetilde{M}_{s_i} \prec ^* \widetilde{\M}$ are already given.

	Continue with an enumeration of $\P(t)$, setting
        $\widetilde{M}_{s_i} := \sigma( \widetilde{M}_{ (s_i \setminus \{n\}) \cup \{n-1\} } )
        \prec ^* \widetilde{M}_t \prec ^* \widetilde{\M}$.
        For the independence condition, we have
        $M_{s_i} \ind _{M_{<s_i}} M_{s_{<i} \cap \P(t)}$ since $s_i$ is part of
        an enumeration of $\P(t)$, and then by transitivity and $M_t
        \ind _{M_{|n-1|}} M_{|n|}$ we deduce $M_{s_i} \ind _{M_{<s_i}} M_{s_{<i}
        \cap \P(t)}M_{|n|}$ and hence $M_{s_i} \ind _{M_{<s_i}} M_{s_{<i}}$.

        Now for the remaining $s_i$: let $M'_{s_i} \prec  \M$ be a constructible
        model over $M_{<s_i} \subseteq  \M$, and let $\widetilde{M}_{s_i}$ be the inverse
        image in $\widetilde{\M}$. The TV Lemma (\factref{fact:TVLemma}) gives $M_{<s_i}
        \subseteq _{TV} M_{s_{<i}}$, so $M_{s_i} \ind _{M_{<s_i}} M_{s_{<i}}$ by
        \lemref{lem:TVAndLIsolation}(iii).
    \end{proof}

    \newcommand{\D}{\Delta}
    \renewcommand{\d}{\operatorname{d}}
    \renewcommand{\L}{\Lambda}
    Define
    \begin{align*}
      \widetilde{\D}       & := \widetilde{M}_{|n|} &
      \widetilde{\D}'      & := \widetilde{M}_{|n+1|} \\
      \d_i\widetilde{\D}   & := \widetilde{M}_{|n| \setminus \{i-1\}} &
      \d_i\widetilde{\D}'  & := \widetilde{M}_{|n+1| \setminus \{i-1\}} \\
      \d\widetilde{\D}     & := \Cup_{1\leq i\leq n} \d_i\widetilde{\D} &
      \d\widetilde{\D}'    & := \Cup_{1\leq i\leq n} \d_i\widetilde{\D}' \\
      \widetilde{\L}       & := \Cup_{1\leq i<n} \d_i\widetilde{\D} &
      \widetilde{\L}'      & := \Cup_{1\leq i<n} \d_i\widetilde{\D}' \\
      \d\d_i\widetilde{\D} & := \Cup_{j\in |n| \setminus \{i-1\}} \widetilde{M}_{|n| \setminus \{i-1,j\} } \\
    \end{align*}

    We also define the corresponding sets in $T$, e.g.\ $\L := \rho(\widetilde{\L}) =
    \Cup_{i<n-1} M_{|n| \setminus \{i\}}$.

    In this notation, the isomorphism of the previous claim is
	\[ \sigma : \widetilde{\D} \xrightarrow{\cong}_{\d_n\widetilde{\D}} \d_n\widetilde{\D}' .\]
    Note that it induces an isomorphism 
	\[ \sigma : \D \xrightarrow{\cong}_{\d_n\D} \d_n\D' .\]

    A diagram for $n=3$:
    
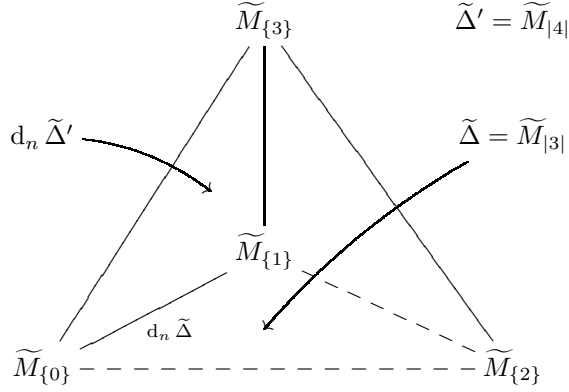
\[
    \xymatrix{
	                                &  & \widetilde{M}_{\{3\}}\ar@{-}[llddd] \ar@{-}[dd] \ar@{-}[rrddd] &  & \widetilde{\D}'=\widetilde{M}_{|4|} \\
	 \d_n\widetilde{\D}' \ar@/^/[rr]+/ld 1cm/  &  &                                                     &  & \widetilde{\D}=\widetilde{M}_{|3|} \ar@/_/[lld]+/d 1cm/ \\
	                                &  & \widetilde{M}_{\{1\}}\ar@{-}[lld]^{\d_n\widetilde{\D}} \ar@{--}[rrd]      &  & \\
	\widetilde{M}_{\{0\}} \ar@{--}[rrrr]       &  &                                                     &  & \widetilde{M}_{\{2\}}
    } \]
    the dashed lines indicate $\widetilde{\L}$, and the faces above them form $\widetilde{\L}'$.

    Let $\widetilde{a} \in \widetilde{\D}$ be a tuple; we want to show that $\tp(\widetilde{a}/\d\widetilde{\D})$ is
    isolated.

    \begin{claim}
	There exists $b_0\in \d_n\D$ such that, setting $A := \acleq(\d\d_n\D b_0)$,
	\[ \tp(\widehat{a}/A\L) \vDash  \tp(\widehat{a}/\d\D). \]
    \end{claim}
    \begin{proof}
	Let $b_0\in \d_n\D$ such that $\tp(a/\d\D) \Dashv  \tp(a/b_0\L)$. First note that every
        extension of $\tp(a_m/b_0\L)$ to $\d\D$ is a non-forking extension.
        Indeed, that holds for $m=1$ by the uniqueness of the extension,
        and hence for any $m$ by interalgebraicity of $a_m$ with $a$.
	So it suffices to see that $\tp(a_m/A\L)$ has a unique non-forking
	extension to $\d\D$. So suppose $c_1,c_2$ realise two such extensions. Then
	$\d_n\D \ind _{A\L} c_i$. Now $\tp(\d_n\D/A)$ is stationary, and since
        the system is independent we have $\d_n\D \ind _{\d\d_n\D} \L$
        and hence $\d_n\D \ind _A A\L$
        , also $\tp(\d_n\D/A\L)$ is stationary. So $c_1 \equiv _{\d_n\D A\L} c_2$,
        so in particular $c_1 \equiv _{\d\D} c_2$,
    \end{proof}
    
    \begin{claim}
	\[ \tp(\widehat{a} / \sigma(\widehat{a}) \L' b_0) \vDash  \tp(\widehat{a} / A\L) \]
    \end{claim}
    \begin{proof}
	Say $\vDash  \phi(a_n,b,e)$ where $b\in A$ and $e\in \L$.

	Say $\theta$ is an algebraic formula isolating $\tp(b/\d\d_n\D b_0)$.

	Let
	    \[ \psi(x) := \forall y\in\theta. (
		\phi(x,y,e) \Leftrightarrow  \phi(\sigma a_n,y,\sigma e) ) , \]
	which is a formula over $\sigma(a_n) \L' b_0$
        since $\sigma e \in \sigma\L \subseteq  \L'$.

	Then $\psi(x) \vDash  \phi(x,b,e)$, since $\vDash  \phi(\sigma a_n,b,\sigma e)$,
	since $b\in \d_n\D$ and $\sigma : \D \xrightarrow{\cong}_{\d_n\D} \d_n\D'$, and
        similarly $\psi(x) \in \tp(a_n / \sigma(a_n) \L' b_0)$.
        So $\tp(\widehat{a} / \sigma(\widehat{a}) \L' b_0) \vDash  \phi(a_n,b,e)$.
    \end{proof}

    Now $\d\widetilde{\D} \subseteq _{TV} \d\widetilde{\D}'$ by the TV lemma, and $\tp(\widetilde{a}/\d\widetilde{\D})$ is
    l-isolated by \lemref{lem:isolLIsol}, so by
    \lemref{lem:TVAndLIsolation}(i),
    $\tp(\widetilde{a} / \d\widetilde{\D}) \vDash  \tp(\widetilde{a} / \d\widetilde{\D}')$.

    Let $\widetilde{b}_0\in \rho^{-1}(b_0) \subseteq  \d_n\widetilde{\D}$, and let $\widetilde{b}_0 \subseteq  \widetilde{b}_0' \in \d_n\widetilde{\D}$
    be such that $\grploc(\widetilde{a}/\d\widetilde{\D})$ is defined over $\widetilde{b}_0'\widetilde{\L}$. Then by
    \lemref{lem:typesKerPres} and the above Claims, we have:
	\[ \tp(\widetilde{a} / \d\widetilde{\D}) \Dashv \vDash  \tp(\widetilde{a} / \sigma(\widetilde{a}) \widetilde{\L}' \widetilde{b}_0'). \]

    So it suffices to see that the latter type is isolated.

    If $n>2$, we have that $\tp(\widetilde{a}\sigma(\widetilde{a})\widetilde{b}_0'/\widetilde{\L}')$ is isolated by
    the inductive hypothesis applied to the $\P(|n-1|)$-$\sim$-system $(\widetilde{M}'_s)_s$
    defined by $\widetilde{M}'_s := \widetilde{M}_{s\cup\{n-1,n\}}$, since $\widetilde{\L}' = \widetilde{M}'_{<|n-1|}$ and
    $\widetilde{M}'_{|n-1|} = \widetilde{M}_{|n+1|} = \widetilde{\D}'$.

    Finally, if $n=2$, we claim that it follows from \lemref{lem:wstab} that
    $\tp(\widetilde{a}\sigma(\widetilde{a})\widetilde{b}_0'/\widetilde{\L}'b_0')$ is isolated. Indeed,
    $\widetilde{\L}'=\widetilde{M}_{\{1,2\}}$, and so it suffices to show that
    $\tp(a,\sigma(a)/M_{\{1,2\}}b_0')$ is isolated, since then for an
    appropriate embedding of the prime model $M_{\{1,2\}}(b_0')$ into $\D'$,
    we have $a,\sigma(a)\in M_{\{1,2\}}(b_0')$.

    We conclude by proving this isolation of
    $\tp(a,\sigma(a)/M_{\{1,2\}}b_0')$. By the definitions of $b_0$ and
    $b_0'$, we have that $\tp(a/b_0'M_{\{1\}})$ implies
    $\tp(a/M_{\{0\}}M_{\{1\}})$ and so is isolated, and hence by
    $M_{\{0,1\}} \ind _{M_{\{1\}}} M_{\{1,2\}}$, also $\tp(a/b_0'M_{\{1,2\}})$ is isolated. Applying $\sigma$, also
    $\tp(\sigma(a)/b_0'M_{\{2\}})$ is isolated, and, applying the TV Lemma and
    \lemref{lem:TVAndLIsolation}(i),
    \begin{align*} \tp(\sigma(a)/b_0'M_{\{2\}}) &\vDash  \tp(\sigma(a)/M_{\{0\}}M_{\{2\}}) \\
    &\vDash  \tp(\sigma(a)/M_{\{0,1\}}M_{\{1,2\}}) \\
    &\vDash  \tp(\sigma(a)/ab_0'M_{\{1,2\}}), \end{align*}
    and so $\tp(a,\sigma(a)/M_{\{1,2\}}b_0')$ is isolated, as required.
\end{proof}

% Note to self: do we win anything over BHHKK by working directly with
% independent systems of finite-dimensional models, rather than having to
% assume a large base? As far as I can see, we don't win anything nice;
% bumping up to a large base is pretty harmless really. That said, I do doubt
% that the atomicity over boundaries just proven holds in the generality of
% BHHKK.

\subsection{Classification}

\begin{lemma}[Constructible Models] \label{lem:constructible}
    Let $M \vDash  T$, let $M_0 \prec  M$ be a copy of the prime model,
    and let $B$ be as in \ssecref{ssec:classificationOutline}.

    Let $\widetilde{M}_0 \vDash  \widehat{T}$ with $\rho(\widetilde{M}_0) = M_0$.

    Then there exists $\widetilde{M} \succ ^* \widetilde{M}_0$ constructible over $B\widetilde{M}_0$,
    with $\rho(\widetilde{M}) = M$.
\end{lemma}
\begin{proof}
    Let $I := \P^{\fin}(B)$.

    Let $(M_s)_{s \in I}$ be a constructible independent $I$-system as given
    by \lemref{lem:baseSys}.

    Let (by \lemref{lem:lPrim}(a)(i))
    $\widetilde{M}$ be an l-constructible model over $M\widetilde{M}_0$, and let $\widetilde{M}_s=\rho^{-1}(M_s) \subseteq  \widetilde{M}$.
    By \lemref{lem:lPrim}(b), $\widetilde{M}_{\emptyset } = \widetilde{M}_0$ and $\rho(\widetilde{M})=M$,
    and by \lemref{lem:inverseStrong}, $(\widetilde{M}_s)_s$ is an $I$-$\sim$-system.

    By \propref{prop:indieAtomic}, $(\widetilde{M}_s)$ is a constructible independent system.
    By \lemref{lem:wstab}, each $\widetilde{M}_{\{p\}}$ for $p \in B$ is constructible
    over $\widetilde{M}_{\emptyset }p$, and so by \lemref{lem:sysPrimOverBasis}, $\widetilde{M}=\widetilde{M}_I$ is
    constructible over $\widetilde{M}_{\emptyset }B = \widetilde{M}_0B$.
\end{proof}

\begin{theorem}[Classification] \label{thm:classification}
    A model $\widetilde{M} \vDash  \widehat{T}$ is determined up to isomorphism among models of $\widehat{T}$ by
    \begin{enumerate}[(i)]\item the isomorphism type of the lift $\widetilde{M}_0 = \rho^{-1}(M_0)$ of a copy $M_0 \prec 
      M$ of the prime model, and
    \item the isomorphism type of $M$ over $M_0$. 
    \end{enumerate}
    
    More explicitly: if $\widetilde{M}^1,\widetilde{M}^2 \vDash  \widehat{T}$, if $\widetilde{M}^1_0 \cong  \widetilde{M}^2_0$ where $\widetilde{M}^i_0$ is the
    lift $\rho^{-1}(M^i_0)$ of a copy $M^i_0 \prec  M^i$ of the prime model, and if
    the induced isomorphism $M^1_0 \cong  M^2_0$ extends to an isomorphism $M^1 \cong  M^2$,
    then $\widetilde{M}^1 \cong  \widetilde{M}^2$, in fact by an isomorphism extending the isomorphism
    $\widetilde{M}^1_0 \cong  \widetilde{M}^2_0$ (but not necessarily agreeing with the isomorphism $M^1 \cong  M^2$).

    \[ \xymatrix{
       \widetilde{M_0} \ar@{^(->}[r] \ar@{->>}[d] & \widetilde{M} \ar@{->>}[d] \\\
       M_0 \ar@{^(->}[r]                          & M \\\
    } \]
\end{theorem}
\begin{proof}
    Given $\widetilde{M} \vDash  \widehat{T}$ and $M_0 \prec  M := \rho(\widetilde{M})$,
    let $B$ be as in \ssecref{ssec:classificationOutline}.

    Then $\widetilde{M}$ is constructible and minimal over $B\widetilde{M}_0$, by
    \lemref{lem:constructible} and the minimality of $M$ over $BM_0$.

    So let $M^i$, $\widetilde{M}^1_0 \cong  \widetilde{M}^2_0$, and $M^1 \cong  M^2$ be as in the statement.
    Let $B^1$ be as in \ssecref{ssec:classificationOutline}, and let $B^2$ be
    the image in $M^2$. Then by the quantifier elimination,
    $B^1\widetilde{M}^1_0 \equiv  B^2\widetilde{M}^2_0$, and by constructibility of $\widetilde{M}^1$ over $B^1\widetilde{M}^1_0$, this
    extends to an elementary embedding $\widetilde{M}^1 \prec  \widetilde{M}^2$; but then by minimality of
    $\widetilde{M}^2$ over $B^2\widetilde{M}^2_0$, the embedding is an isomorphism.
\end{proof}

\begin{remark} \label{rem:homog}
    We can also conclude that if $M$ is strongly $\aleph_1$-homogeneous
    (e.g.\ if we take $M$ to be saturated and uncountable), then $\widetilde{M}$ is
    strongly $\aleph_0$-homogeneous over $\widetilde{M}_0$. 
    Indeed, if $\widetilde{a} \equiv _{\widetilde{M}_0} \widetilde{a}'$, then by homogeneity we have $B'$ such that
    $B \widehat{a} \equiv _{M_0} B' \widehat{a}'$, so $B\widetilde{M}_0\widetilde{a} \equiv  B'\widetilde{M}_0\widetilde{a}'$; but $\widetilde{M}$ is also
    constructible and minimal over $B\widetilde{M}_0\widetilde{a}$ and over $B'\widetilde{M}_0\widetilde{a}'$, so this
    extends to an automorphism.

    Similarly, we obtain strong $\aleph_0$-homogeneity over an arbitrary
    countable *-submodel $\widetilde{M}_1 \prec ^* \widetilde{M}$, replacing $B$ with $\acl$-bases over
    $M_1$.

    Moreover, by \propref{prop:indieAtomic}, we similarly obtain strong
    $\aleph_0$-homogeneity over $\widetilde{M}_{<|n|}$ for a $\P(|n|)$-$\sim$-system in $\widetilde{M}$.
    Note that in the context of \cite{BHHKK}, and even in the specific example
    of pseudo-exponentiation, the corresponding results require a saturation
    hypothesis on $\widetilde{M}_s$.
\end{remark}

\section{Exponential maps of semiabelian varieties}
\label{sec:abelian}
In this section, we apply our classification result
\thmref{thm:classification}, along with some arithmetic Kummer theory, to
prove \thmref{thm:catAbelian} and draw some related conclusions.

We actually work in slightly greater generality than \thmref{thm:catAbelian},
by allowing split semiabelian varieties.
So throughout this section, we will suppose that $\G(\C)$ is
the product $A\times \G_m^n$ of a (possibly trivial) complex abelian variety
and a (possibly trivial) algebraic torus.

Let $\er := \End(\G)$ be the ring of algebraic endomorphisms of $\G$. Suppose
$\G$ and its endomorphisms are defined over $k_0 \leq  \C$.

We first explain how we attach to the algebraic group $\G$ a theory $T$
satisfying the assumptions of the previous sections.

$\G$ can be viewed as a definable group in $\operatorname{ACF}_0$, and
as such inherits the structure of a finite Morley rank group. 
Explicitly, we consider $\G(K)$, for $K$ an algebraically closed extension of
$k_0$, as a structure in the language $L$ consisting of a predicate for each
$k_0$-Zariski-closed subset of each Cartesian power $\G^n(K)$. This structure
is bi-interpretable with the field $(K;+,\cdot,(c)_{c\in k_0})$ with
parameters for $k_0$, and is a finite Morley rank group of rank $\dim(\G)$.
We let $T$ be the theory of $\G(\C)$ in the language consisting of a predicate
for each $k_0$-Zariski-closed subset of $\G^n(\C)$. This is a commutative
divisible group of finite Morley rank, admits quantifier elimination, and, by
\lemref{lem:subgroupsEndomorphisms} below, every connected definable subgroup
of $\G^n$ is over $k_0$, so is defined over $\emptyset $ in $T$.

\subsection{$\er$-module homomorphisms as models of $\widehat{T}$}
By \propref{prop:analMod}, the Lie exponential map $\exp : L\G \twoheadrightarrow  \G(\C)$
has the structure of a model of $\widehat{T}$, which we denote by $L\G$. As a step towards
proving \thmref{thm:catAbelian}, we prove in this subsection an abstract
algebraic version of this.

Let $\er := \End(\G)$ be the ring of algebraic (equivalently, definable)
endomorphisms. The derivative at the identity $L\eta$ of $\eta\in\er$ is a
$\C$-linear endomorphism of $L\G$, and we consider $L\G$ as an $\er$ module
with this action.

\begin{lemma} \label{lem:subgroupsEndomorphisms}
  \begin{enumerate}[(i)]\item Any connected algebraic subgroup $H \leq  \G^n$ is the connected component of
      the kernel of an endomorphism $\eta\in \End(\G^n) \isom \operatorname{Mat}_{n,n}(\er)$, and
  \item $LH \leq  L\G^n$ is then the kernel of $L\eta \in \End_\C(L\G^n)$.
  \end{enumerate}
\end{lemma}
\begin{proof}
  \begin{enumerate}[(i)]\item By Poincar\'e's complete reducibility theorem, there exists an algebraic
      subgroup $H'$ such that the summation map $\Sigma:H\times H' \maps \G^n$ is
      an isogeny. So say $\theta : \G^n \maps H\times H'$ is an isogeny such that
      $\theta\Sigma=[m]$, and let $\pi_2 : H\times H' \maps H'$ be the
      projection. Then $\pi_2\theta\Sigma(h,h')=mh'$, so
      $\ker(\pi_2\theta)^o=\Sigma(H\times H'[m])^o=(H+H'[m])^o=H$.
  \item $L\eta$ takes values in the discrete group $\ker(\exp)^n$ on $LH$, so
      by connectedness and continuity $L\eta$ is zero on $LH$. Conversely,
      $\exp(\ker(L\eta))$ is a divisible subgroup of $\ker(\eta)$, and hence is
      contained in $\ker(\eta)^o = H$. So $\ker(L\eta)$ is a subgroup of
      $\exp^{-1}(H)$ containing $LH$; but $\ker(L\eta)$ is a $\C$-subspace so is
      connected, so $\ker(L\eta)=LH$.
  \end{enumerate}
\end{proof}
\begin{remark}
    \lemref{lem:subgroupsEndomorphisms} (i) can fail for $\G$ a semiabelian variety.
\end{remark}

\begin{proposition} \label{prop:OModMods}
    If $K$ is an algebraically closed field extension of $k_0$, any
    surjective $\er$-module homomorphism $\rho : V \twoheadrightarrow  \G(K)$ from a
    divisible torsion-free $\er$-module $V$ with finitely generated kernel is
    a model of $\widehat{T}$, where $\widehat{H}$ is interpreted as the kernel of the action of
    $\eta$ on $V^n$ if $H$ is the connected component of the kernel of
    $\eta\in\End(\G^n) \cong  \operatorname{Mat}_{n,n}(\er)$, and $\rho_n(x) := \rho(x/n)$.
\end{proposition}
\begin{proof}
    We appeal to \lemref{lem:homMod}.
    We will see in the course of the proof that $\widehat{H}$ is indeed the divisible
    part of $\rho^{-1}(H)$, as assumed in that lemma, hence in particular that
    $\widehat{H}$ is well-defined.
    
    We use the following elementary principle, which we will call (*):
    if $A,B,F$ are subgroups of a torsion-free abelian group, $A$ and $B$ are
    divisible, and $F$ is finitely generated, and if $A \leq  B + F$, then $A \leq 
    B$.

    Suppose $H = \ker(\eta)^o \leq  \G^n$, and $\widehat{H} = \ker(\eta)$.
    We show that $\rho_k(\widehat{H})=H$ for all $k$.
    By working in $\G^n$, we may assume $n=1$.
    Let $\Lambda := \ker\rho \leq  V$,
    and let $\Lambda_0 \leq  V$ be the divisible hull of $\Lambda$.
    \begin{claim}
        $\eta(\Lambda_0) = \im\eta\cap\Lambda_0$.
    \end{claim}
    \begin{proof}
        First, note that $\eta(\Tor(\G)) = \Tor(\im\eta)$.
        Indeed, if $n\eta(g) = 0$, then $ng \in \ker\eta$, so $mng \in
        (\ker\eta)^o$ for some $m$; then by divisibility of $(\ker\eta)^o$,
        say $h \in (\ker\eta)^o$ with $mnh=mng$. Then $\eta(g-h) = \eta(g)$ and
        $g-h \in \Tor(\G)$.

        Hence $\im\eta\cap\Lambda_0 \leq  \eta(\Lambda_0) + \Lambda$,
        so by (*) already $\im\eta\cap\Lambda_0 \leq  \eta(\Lambda_0)$. The
        converse is immediate.
    \end{proof}
    Now since $\Lambda$ is finitely generated, $\eta(\Lambda)$ is a finite
    index subgroup of $\im\eta \cap \Lambda$. By the snake lemma (see
    diagram), it follows that $\rho(\widehat{H})$ is of finite index in $\ker(\eta)$.
    So by divisibility of $\widehat{H}$, we have $\rho(\widehat{H})=\ker(\eta)^o=H$, and
    then $\rho_k(\widehat{H})=\rho(\widehat{H})=H$ for all $k$. So \axref{ax:epic} holds. 

    \[ \xymatrix{
                  &                       & \Lambda \ar[r]\ar[d] & \Lambda\cap\im\eta \ar[r]\ar[d] & \cdots \\
                  & \widehat{H}       \ar[r]\ar[d] & V \ar[r]\ar[d]_\rho  & \im\eta \ar[r]\ar[d]            & 0\\
    0 \ar[r]      & \ker\eta \ar[r]\ar[d] & \G \ar[r]\ar[d]      & \im\eta \\
    \cdots \ar[r] & \operatorname{Finite} \ar[r]       & 0
    } \]

    Hence $\rho^{-1}(H) = \widehat{H} + \Lambda$,
    so by (*), $\widehat{H}$ is the divisible part of $\rho^{-1}(H)$.

    Finally, we verify \axref{ax:proj}. Let $\pr : G \twoheadrightarrow  H$ be as in that
    axiom.
    By \axref{ax:epic}, 
        \[ \rho(\pr(\widehat{G})) = \pr(\rho(\widehat{G})) = \pr(G) = H = \rho(\widehat{H}) ,\]
    so $\pr(\widehat{G}) + \Lambda = \widehat{H} + \Lambda$,
    so by (*), $\pr(\widehat{G}) = \widehat{H}$.
\end{proof}

% TODO: what about the relative case:
% if we work with meromorphic germs (or functions on a polydisc), the
% relative exp is surjective with f.g. kernel. Presumably we also have
% ax:epic, so the above proof will go through?

\subsection{Kummer theory}
Suppose now that $k_0$ is a number field.

Using this assumption, we may appeal to Kummer theory to reduce consideration
of the prime model to consideration of the kernel. This is essentially the
same argument as in \cite[Lemma 4]{MishaTrans}.

Recall $T = \Th(\G(\C))$ in the language with a predicate for each subvariety
defined over $k_0$ of a cartesian power of $\G$.

\begin{lemma} \label{lem:kummer}
    Suppose $\widetilde{M}_0 \vDash  \widehat{T}$ with $M_0 = \rho(\widetilde{M}_0) = \G(\Qbar)$, the prime model
    of $T$. Then $\widetilde{M}_0$ is constructible over $\ker(\widetilde{M}_0)$.
\end{lemma}
\begin{proof}
    Write $\ker$ for $\ker(\widetilde{M}_0)$.

    We use notation and results from \secref{sec:kummer}.

    By \lemref{lem:atomCons}, it suffices to show atomicity. Let $\widetilde{c} \in \widetilde{M}_0$.

    Let $H+\zeta$ be the minimal torsion coset (see \secref{sec:openness})
    containing $c$. Then $\widetilde{c} \in \widehat{H}+\widetilde{\zeta}$ for some $\widetilde{\zeta} \in \Q\ker$.
    By translating, we may assume $\widetilde{\zeta}=0$, so then $\widetilde{c} \in \widehat{H}$ and $H$ is
    the minimal torsion coset containing $c$.

    % In fact, in this situation we then have that \grploc(\widetilde{c}/\ker) = \widehat{H}. But
    % we only use "\leq ", which in greater generality (if we included vectorial
    % parts) might be the best we could get.

    By \propref{prop:kummer}, the image of the Kummer pairing is open,
        \[ Z_\infty := \pairing{\Gal(k_0(c,\G[\infty]))}{c} \leq _{\operatorname{op}} T_\infty^{H} ,\]
    so $nT_\infty^{H} \leq  Z_\infty$ for some $n$,
    so 
    \[ \tp^{T}(c_n/\G[\infty]) \cup \Cup_i \{ c_i \in H \} \cup
    \Cup_{k,m} \{ [m] c_{km} = c_k \} \vDash  \tp^{T}(\widehat{c}/\G[\infty]) .\]

    So since $\widehat{\ker} = G[\infty]$, it follows by \lemref{lem:typesKerPres} that
        \[ \tp^{T}(c_n/\G[\infty]) \cup \{\widetilde{c} \in \grploc(\widetilde{c}/\ker) \} \vDash  \tp(\widetilde{c}/\ker) .\]
    But $\tp^{T}(c_n/\G[\infty])$ is isolated since $c_n \in \G(\Qbar)$, so
    $\tp(\widetilde{c}/\ker)$ is isolated as required.
\end{proof}

% note: although I'm tempted to try every time I read it, please note that
% this does *not* appear to admit any nice formulation in terms of
% constructibility: we need the prime model downstairs, and then \widetilde{M}_0 is obviously
% not constructible over M_0\ker, and we'd need all of \widetilde{M}_0 upstairs.

\subsection{Categoricity and characterisation}
We continue to assume that $k_0$ is a number field.

Combining \lemref{lem:kummer} with \thmref{thm:classification}, and using
uncountable categoricity of $T$ to simplify the latter, we conclude:
\begin{theorem} \label{thm:classAbelian}
    A model $\widetilde{M}$ of $\widehat{T}$ is determined up to isomorphism over $\ker(\widetilde{M})$ by
    \begin{enumerate}[(i)]\item the isomorphism type of $\ker(\widetilde{M})$, equipped with all structure induced from $\widehat{T}$
    \item the transcendence degree of $K_M$, where $M \cong  \G(K_M)$.
    \end{enumerate}
\end{theorem}
\begin{proof}
    Suppose $\widetilde{M}^1,\widetilde{M}^2 \vDash  \widehat{T}$,
    and $\ker(\widetilde{M}^1) \cong  \ker(\widetilde{M}^2)$ and $\trd(K_{M^1}) = \trd(K_{M^2})$.
    Let $\widetilde{M}^i_0$ be the inverse image of $M^i_0 := \G(\Qbar) \prec  M^i$.
    Then by \lemref{lem:kummer} and the minimality of $\G(\Qbar)$ over $\emptyset $,
    the isomorphism $\ker(\widetilde{M}^1) \cong  \ker(\widetilde{M}^2)$ extends to an
    isomorphism $\widetilde{M}^1_0 \cong  \widetilde{M}^2_0$. The induced isomorphism $M^1_0 \cong  M^2_0$
    extends to an isomorphism $M^1 \cong  M^2$; indeed, it induces a field
    automorphism of $\Qbar$ over $k_0$, which by the equality of transcendence
    degrees extends to an isomorphism $K_{M^1} \cong  K_{M^2}$, inducing an
    isomorphism $M^1 \cong  M^2$.

    We conclude by \thmref{thm:classification}.
\end{proof}

\begin{corollary} \label{cor:categoricity}
  The model $L\G \vDash  \widehat{T}$ is the
  unique $\widehat{L}$-structure $\widetilde{M}$ satisfying:
    \begin{enumerate}[(I)]\item $\widehat{T}$
    \item $|\widetilde{M}| = 2^{\aleph_0}$
    \item $\ker^{\widetilde{M}} \isom \ker^{L\G}$, a partial $\widehat{L}$-isomorphism.
    \end{enumerate}

    Moreover, for any such $\widetilde{M}$, the isomorphism of (III) extends to an
    isomorphism of $\widetilde{M}$ with $L\G$.
\end{corollary}

\begin{theorem}[\thmref{thm:catAbelian}]
  Suppose $\rho,\rho' : L\G \twoheadrightarrow  \G(\C)$ are surjective $\er$-module homomorphisms,
  $\ker\rho'=\ker\rho$, and
  $\rho'\restricted_{\spanofover{\ker\rho'}{\Q}} =
  \rho\restricted_{\spanofover{\ker\rho}{\Q}}$.

  Then there exists an 
  $\er$-module automorphism $\sigma\in\Aut_\er(L\G/\ker\rho)$ and a field
  automorphism $\tau\in\Aut(\C/k_0)$ of $\C$ fixing $k_0$ such that
  $\tau\rho'=\rho\sigma$.
\end{theorem}
\begin{proof}
  Let $\widetilde{M}$ and $\widetilde{M}'$ be the corresponding $\widehat{L}$ structures.
  By \propref{prop:OModMods}, they are models of $\widehat{T}$.
  By the QE and the assumption on the kernels, the structure induced on
  $\ker\rho$ by the two structures is the same, and the transcendence degrees
  are both $2^{\aleph_0}$. So by
  \thmref{thm:classAbelian}, $\widetilde{M}' \cong  \widetilde{M}$ as $\widehat{L}$-structures,
  by an isomorphism fixing $\ker\rho$.
  Since the graphs of addition and of the action of each $\eta\in\er$ are
  interpretations of appropriate $\widehat{H}$, this isomorphism induces an
  $\er$-module automorphism $\sigma$ of $L\G$, and by the choice of language
  for $T$ it induces a field automorphism $\tau$ over $k_0$.
  % and actually k_\infty\ldots worth noting?
\end{proof}

Understanding the structure of $\ker$ involves an understanding of the action
of Galois on the torsion, which in general is known to be a hard problem.
But let us highlight a strengthening of \thmref{thm:classAbelian} in the case
of the characteristic 0 multiplicative group:
\begin{theorem}
  Let $\G=\G_m(\C)$. Then a model $\widetilde{M}$ of $\widehat{T}$ is determined up to isomorphism
  by the transcendence degree of the algebraically closed field $K$ such that
  $M \cong  \G_m(K)$, and the isomorphism type of $\ker\rho$ as an abstract
  group.
\end{theorem}
\begin{proof}
  This is immediate from \thmref{thm:classAbelian} once we see that the
  isomorphism type of $\ker$ as a $\widehat{L}$-structure is
  determined by its isomorphism type as an abstract group. But this follows
  easily from the quantifier elimination and the fact from cyclotomic theory
  that any group automorphism of the roots of unity is a Galois automorphism.
\end{proof}

\begin{remark}
  In the case of an elliptic curve $\G = E$ there are only finitely many
  isomorphism types for a kernel with underlying group $\seq{\Z^2;+}$
  (\cite{MishaTrans}, \cite[Theorem~4.3.2]{BaysThesis}).

  See also \cite[IV.6.3,IV.7.4]{GavThesis} for some discussion of the higher
  dimensional situation.
\end{remark}

\begin{question}
  The assumption that $k_0$ is a number field was used in \lemref{lem:kummer}.
  It is natural to ask whether this is essential. Does an appropriate
  version of Kummer theory go through for Abelian varieties over function
  fields? We are not aware of this question being fully addressed in the
  literature, but \cite[Theorem~5.4]{BertrandLuminy} goes some way toward it.
\end{question}

\section{Further examples}
\label{sec:examples}
In this section, we make some brief remarks on some other natural examples of
\thmref{thm:classification}.

\subsection{Positive characteristic}
We can not in general expect to improve on \thmref{thm:classification} in
positive characteristic: if $\G$ is the multiplicative group of a
characteristic $p>0$ algebraically closed field, then the prime model is
$G(\GF_p^{\operatorname{alg}})$, which is also the torsion group of $\G$. In this case, we
recover the main theorem, 2.2, of \cite{BZCovers}.

\subsection{Manin kernels}
In the theory $\operatorname{DCF}$ of differentially closed fields of characteristic 0, the
Kolchin closures of the torsion of semiabelian varieties, also known as
{\em Manin kernels}, are commutative divisible groups of finite Morley rank.
A connected definable subgroup of such a Manin kernel is the Manin kernel of
its Zariski closure,
% using BBP 4.9; see maninCover. But probably better not to cite here.
so Manin kernels of semiabelian varieties are rigid.
Our classification theorem therefore applies to this case. By considering a
local analytic trivialisation, a natural analytic model of $\widehat{T}$ for $\G$ a
(non-isoconstant) Manin kernel can be given; this will be addressed in future
work.

\subsection{Meromorphic Groups}
\label{ssec:meroGrp}
Let $\G$ be a connected meromorphic group in the sense of
\cite{PSMeroGrp}, i.e.\ a connected definable group in the structure
$\mathcal{A}$ of compact complex spaces definable over $\emptyset $ (equivalently,
over $\C$). By \cite[Fact~2.10]{PSMeroGrp}, $\G$ can be uniquely
identified with a complex Lie group.

Considering $\G$ with its induced structure, it is a finite Morley rank group.
Suppose $\G$ is commutative and rigid. By the classification in
\cite{PSMeroGrp} and the fact that any commutative complex linear
algebraic group is a product of copies of $\G_m$ and $\G_a$, there is a
definable exact sequence of Lie groups
    \[ 0 \rightarrow  \G_m^n \rightarrow  \G \rightarrow  H \rightarrow  0 \]
where $H$ is a complex torus.  It is also shown in \cite{PSMeroGrp} that
$\G$ is definable in a K\"{a}hler space; the latter may be considered in a
countable language by \cite{MoosaKahler}, so we may consider the language of
$\G$ to be the induced countable language. Let $T=\Th(\G)$. 

In particular, in the case that $\G$ is a complex semiabelian variety, we may
take the language to be that induced from the field, as in
\secref{sec:abelian} above.

Now let $L\G$ be the analytic universal cover of the Lie group $\G$,
considered as an $\widehat{L}$-structure as in \ssecref{ssec:lie}.

By \propref{prop:analMod}, $L\G \vDash  \widehat{T}$.
So by \thmref{thm:classification}, $L\G$ is the unique kernel-preserving
extension of its restriction to the prime model $\G_0$ of $\G$, which is a
countable structure.

\begin{question}
    Could the Kummer theory of Lemma \ref{lem:kummer} apply here?
    Concretely: is $\rho^{-1}(\G_0)$ atomic over $\ker$? 
\end{question}

%$\C^g/\Lambda(t)$, always an abelian variety (else should work in "DCCCM"??)

% Green fields: obvious candidate for a model is an appropriately collapsed
% version of L+Q < \C where L is an appropriately angled real line. But I'm
% not sure how to handle the collapse.

\appendix
\section{Kummer theory for $A\times\G_m^n$}
\label{sec:kummer}

In this appendix, we show that the results on Kummer theory for abelian
varieties over number fields apply also to semiabelian varieties of the form
$A\times\G_m^n$ for $A$ an abelian variety over a number field.
This should perhaps be considered a known result, but we could find no
complete proof in the literature. 

Our approach owes much to Daniel Bertrand. In the case that $\G=A$, the Kummer
theoretical result we require is precisely \cite[Theorem~5.2]{BertrandLuminy};
the purpose of this appendix is to show that this result holds also for
$A\times\G_m^n$, with a mostly parallel proof. As in that article, the method we
apply is essentially that of Ribet's paper \cite{RibetKummer}.

We should note that for general semiabelian varieties over number fields,
Kummer theory is known to fail due to the existence of {\em deficient points}
- see \cite{JacquinotRibet84}.

\subsection{Finiteness theorems for abelian varieties}
Let $A$ be an abelian variety over a number field $k_0$,
let $T^{A}_l := \invlim_n A[l^n]$ for $l$ prime be the Tate modules, and let
$T^{A}_\infty := \invlim_n A[n] = \Pi_l T^{A}_l$.

The following result on Galois cohomology is a consequence of Serre's uniform
version of Bogomolov's result on homotheties. Here and below, $H^i$ refers to
continuous group cohomology.
\begin{fact} \label{fact:serre}
    $H^1(\Gal(k_0(A[\infty])/k_0),A[n])$ has uniformly bounded finite exponent,
    i.e.\ there exists $c>0$ such that for all $n>0$,
    we have $c\cdot H^1(\Gal(k_0(A[\infty])/k_0),A[n]) = 0$.
    % Note: the same proof gives that $H^1(\Gal(k_0(A[\infty])/k_0),T_\infty)$
    % has finite exponent p\widehat{M}-1, but without some control on H^2 this seems to
    % be a weaker statement.
\end{fact}
\begin{proof}
    Let $G_\infty := \Gal(k_0(A[\infty])/k_0)$.

    Note that $H^1(G_\infty,A[n])$ admits a prime power decomposition as
    $\prod_i H^1(G_\infty,A[l_i^{k_i}])$ where $n=\prod_i l_i^{k_i}$.

    By
    \cite[Th\'eor\`eme 2', ``R\'esum\'e des cours de 1985-1986'', proved in
    ``Lettre \'a Ken Ribet du 7/3/1986'' in the same volume]{SerreOeuvresIV},
    there exists $M>0$ such that every $M$th power homothety is in the image of
    $G_\infty$, i.e.\ any element of $\Zhat^* = \Pi_l \Z_l^*$ which is an $M$th
    power in that group is the action on $T^{A}_\infty$ of some element of
    $G_\infty$.
    
    In particular, there is $\sigma \in G_\infty$ which acts on $T^{A}_l$ as
    multiplication by $2^{M}$ for $l \neq  2$, and acts on $T^{A}_2$ as the
    identity. Then $\sigma$ is central in $G_\infty$, so by Sah's Lemma,
    $H^1(G_\infty,T^{A}_\infty)$ and each $H^1(G_\infty,A[n])$ are annihilated
    by $\sigma-1$. Then if $l$ is an odd prime which does not divide $2^{M}-1$,
    so $2^{M}-1 \in \Z_l^*$, we have $H^1(G_\infty,A[l^k])=0$ for all $k$.

    Let $2=l_0,l_1,\ldots,l_s$ be the remaining primes, and let $p \notin \{l_0,\ldots,l_s\}$
    be another prime. Then by the same argument, $p^{M}-1$ annihilates
    each $H^1(G_\infty,A[l_i^k])$.

    So $p^{M}-1$ annihilates each $H^1(G_\infty,A[n])$.
\end{proof}

The second ingredient is the following result of Faltings, sometimes referred
to, after Lang, as Finiteness I \cite[IV.2]{LangSurveyDioph}. Here, a
\defnstyle{$k_0$-isogeny} is an isogeny defined over $k_0$; similarly for
\defnstyle{$k_0$-isomorphism}.
\begin{fact}[Faltings] \label{fact:faltings}
    The algebraic groups which are $k_0$-isogenous to
    $A$ fall into finitely many $k_0$-isomorphism classes.
\end{fact}

\subsection{Generalisations to $A \times \G_m^n$}
Let $\G = A \times \G_m^n$ with $A$ an abelian variety over a number field
$k_0$. We check that the results of the previous section imply the
corresponding results for $\G$.

\begin{lemma} \label{lem:serre}
    $H^1(\Gal(k_0(\G[\infty])/k_0),A[n])$ has uniformly bounded finite exponent,
    i.e.\ there exists $c>0$ such that for all $n>0$,
    we have $c\cdot H^1(\Gal(k_0(\G[\infty])/k_0),A[n]) = 0$.
\end{lemma}
\begin{proof}
    By Hilbert 90, $H^1(\Gal(k_0(\G[\infty])/k_0),\mu_m) = 0$.
    Meanwhile, $k_0(\G[\infty]) = k_0(A[\infty])$ since the multiplicative
    roots of unity are rational over $k_0(A[\infty])$, via a Weil pairing.

    So
      \begin{align*} H^1(\Gal(k_0(\G[\infty])/k_0),\G[n])
        &\cong  H^1(\Gal(k_0(\G[\infty])/k_0),A[n]) \\
        &= H^1(\Gal(k_0(A[\infty])/k_0),A[n]) ,
        \end{align*}
    and we conclude by \factref{fact:serre}.
\end{proof}

\begin{lemma} \label{lem:faltings}
    The algebraic groups which are $k_0$-isogenous to
    $\G$ fall into finitely many $k_0$-isomorphism classes.
\end{lemma}
\begin{proof}
    \providecommand{\dual}[1]{{#1}^\vee}
    Let $T := \G_m^n$.

    Recall (see e.g.\ \cite[10]{SerreAbelian}) that a semiabelian variety
    which falls into an exact sequence $0 \rightarrow  T \rightarrow  S \rightarrow  A \rightarrow  0$
    corresponds to a point in the $n$th power of the dual abelian variety of
    $A$,
    \[ \Ext(A,T) \cong  \Ext(A,\G_m)^n \cong  (\dual{A})^n .\]

    Let $\G'$ be $k_0$-isogenous to $\G$, so $\G' \cong  \G/Z$
    for $Z \leq  \G$ a finite subgroup defined over $k_0$. Since $\G / (Z\cap T)$
    is $k_0$-isomorphic to $\G$, we may assume $Z\cap T = 0$.

    Let $\pi_1 : \G \rightarrow  A$ and $\pi_2 : \G \rightarrow  T$ be the projections of the product.
    Let $A' := A / \pi_1(Z)$ be the quotient abelian variety. So $\G'$ is an
    extension of $A'$ by $T$, and so $\G'$ corresponds to an element $[\G']$ of
    $\Ext(A',T) \cong  (\dual{A'})^n$.

    \begin{claim}
        $[\G']$ is a torsion element of $\Ext(A',T)$.
    \end{claim}
    \begin{proof}
      Let $k$ be the exponent of the finite group $\pi_2(Z) \leq  T$.
      %Working analytically, it's not hard to see that the "factors of
      %automorphy" involved (Birkenhage-Lange) are constant and torsion, i.e.\ a
      %homomorphism in $\Hom(\pi_1(A),T[k])$, so $\G'$ is torsion in
      %$\Adual^n$.
      %But we can also see it directly.
      Then the $k$-fold Baer sum $[k]\G'$ of $\G'$ in $\Ext(A',T)$ is split.
      Indeed, $[k]\G'$ is the $k$-fold fibre product of $\G'$ over $A'$,
      quotiented by the subgroup
      $\Sigma := \{ \Sigma_i \alpha_i = 0 \;|\; \alpha_i \in T \} \leq  T^k \leq  A'^k$.
      Then the trivialisation $x \mapsto  (x,0)$ of $\G = A \times T$ induces a
      trivialisation of $[k]\G'$,
      $x + \pi_1(Z) \mapsto  ((x,0)+Z, \ldots, (x,0)+Z) + \Sigma$;
      this is well-defined as $((x,0)+Z) - ((x+\pi_1\zeta,0)+Z) =
      (0,\pi_2\zeta)+Z$,
      and $(\pi_2\zeta, \ldots, \pi_2\zeta) \in \Sigma$ since $k\pi_2\zeta=0$.
    \end{proof}

    Now since $\G'$ is defined over $k_0$, so is $A'$ and so is the torsion
    point $[\G']$ of $(\dual{A'})^n$. By \factref{fact:faltings},
    there are only finitely many such $A'$ up to $k_0$-isomorphism, and by
    Mordell-Weil each has only finitely many $k_0$-rational torsion points.
    Hence, there are only finitely many possibilities for $\G'$ up to
    $k_0$-isomorphism.
\end{proof}

\subsection{Group structure of $\G(k_0(\G[\infty]))$}
\begin{definition}
    If $\Gamma'$ is a subgroup of an abelian group $\Gamma$, let
    $\pureHull_{\Gamma}(\Gamma') := \{\gamma \in \Gamma \;|\; \exists
    n>0 \qsep n\gamma\in\Gamma'\} \leq  \Gamma$.

    An abelian group $\Gamma$ is {\em locally free modulo torsion} if for any
    finitely generated subgroup $\Gamma' \leq  \Gamma$, there exists $m$ such
    that $m\cdot\pureHull_{\Gamma}(\Gamma') \leq  \Gamma' + \Tor(\Gamma)$.
\end{definition}

Now let $k_0$ be a number field,
let $A$ be an abelian variety over $k_0$,
and let $\G = A \times \G_m^n$ be the product with an algebraic torus.
Let $k_\infty := k_0(\G[\infty])$.

\begin{lemma} \label{lem:boundedDiv}
    $\G(k_\infty)$ is locally free modulo torsion.
\end{lemma}
\begin{remark}
    By countability of $\G(k_\infty)$ and a theorem of Pontryagin
    \cite[19.1]{FuchsInfAb}, an equivalent statement is that the quotient group
    $\G(k_\infty) / \G[\infty]$ is free abelian. For $\G$ an abelian
    variety over a number field, this is proven by Larsen in \cite{LarsenMWTor}.
    This lemma generalises that result, using similar techniques.
\end{remark}
\begin{proof}
    Let $\Gamma \leq  \G(k_\infty)$ be a finitely generated subgroup.
    Replacing $k_0$ by the number field $k_0(\Gamma)$ if necessary, we
    assume $\Gamma \leq  \G(k_0)$. 

    First, we see that $\G(k_0) = A(k_0) \times \G_m^n(k_0)$ is free modulo
    torsion.
    We use Dirichlet's Unit theorem to examine the group structure of
    $\G_m(k_0) = k_0^*$. Here, we are following \cite[Lemma 2.1]{ZCovers}.

    \providecommand{\valring}{\mathcal{O}}
    Let $\valring_{k_0}$ be the ring of integers of $k_0$. By
    Dirichlet's Unit theorem, $\valring_{k_0}^*$ is finitely generated.
    Recall that $\valring_{k_0}$ is a Dedekind domain and the fractional ideals,
    $\operatorname{Id}(\valring_{k_0})$, form a free abelian group with generators the prime
    ideals. We have an exact sequence
    \[ \xymatrix{ 1 \ar[r] & \valring_{k_0}^* \ar[r] & k_0^* \ar[r]^-\theta &
    \operatorname{Id}(\valring_{k_0}) } ,\]
    where $\theta(x) := x\valring_{k_0}$. The image of $\theta$ is a subgroup
    of a free abelian group, so is free abelian.

    Meanwhile, $A(k_0)$ is finitely generated by the Mordell-Weil theorem.
    So $\G(k_0)$ is an extension of a free abelian group by a finitely
    generated group, so the quotient by the torsion is an extension of free
    abelian by free abelian, so is free abelian. Hence $\G(k_0)$ is locally
    free modulo torsion.

    So say $m$ is such that
    $m\cdot\pureHull_{\G(k_0)}(\Gamma) \leq  \Gamma + \G[\infty]$.
 
    Meanwhile, by \lemref{lem:serre},
    say $c \cdot H^1(\Gal(k_\infty/k_0),\G[n]) = 0$ for all $n$.

    We conclude by showing
    $mc\cdot\pureHull_{\G(k_\infty)}(\Gamma) \leq  \Gamma + \G[\infty]$.

    Indeed, suppose $\gamma \in \pureHull_{\G(k_\infty)}(\Gamma)$,
    say
    $\gamma \in \G(k_\infty)$ and $n\gamma \in \Gamma \leq  \G(k_0)$.
    Then $\theta(\sigma) := \sigma\gamma - \gamma$ yields an element of
    $H^1(\Gal(k_\infty/k_0),\G[n])$.
    So $c\theta$ is a coboundary,
    so there is $\zeta \in \G[n]$ such that
    $c(\sigma\gamma - \gamma) = \sigma\zeta-\zeta$ for all
    $\sigma \in \Gal(k_\infty/k_0)$,
    so $c\gamma - \zeta \in \G(k_0)$.

    So $c\gamma - \zeta \in \pureHull_{\G(k_0)}(\Gamma)$,
    so $mc\gamma \in \Gamma + \G[\infty]$.
\end{proof}

\subsection{Openness}
\label{sec:openness}
Let $\G = A \times \G_m^n$ as above.
Let $\er := \End(\G) \cong  \End(A) \times \End(\G_m^n)$.
By taking a finite field extension if
necessary, we assume that each $\eta \in \er$
is defined over the number field $k_0$.

We define the Kummer pairings for $\G$ as follows: if $k \geq  k_0$, and $\gamma
\in \G(k)$ and $\sigma\in\Gal(k)$, let $\pairing{\sigma}{\gamma}_n :=
\sigma\alpha - \alpha \in \G[n]$ for any $\alpha \in \G(\bar{k})$ with
$n\alpha=\gamma$, and let $\pairing{\sigma}{\gamma} :=
(\pairing{\sigma}{\gamma}_n)_n \in T_\infty^{\G}$.

A {\em torsion coset} in $\G$ is the translate $H+\zeta$ of a connected algebraic
subgroup $H \leq  \G$ by a torsion point $\zeta \in \G[\infty]$.

By considering the torsion group, one sees that $T^{H}_\infty$ for such an $H$
is isomorphic to a finite power of $\Zhat$, and so a subgroup $Z$ of
$T^{H}_\infty$ is open in the profinite topology, $Z \leq _{\operatorname{op}} T^{H}_\infty$,
if and only if it is of finite index.

\begin{proposition} \label{prop:kummer}
    Let $\gamma \in \G(k_\infty)$. Suppose $H+\zeta$ is the minimal torsion
    coset containing $\gamma$.
    Then
    $Z_\infty := \pairing{\Gal(k_\infty)}{\gamma} \leq _{\operatorname{op}} T^{H}_\infty \leq 
    T^{\G}_\infty$.
\end{proposition}
\begin{remark}
    In the case that $\G$ is an abelian variety, this is exactly
    \cite[Theorem~5.2]{BertrandLuminy}.
\end{remark}
\begin{proof}
    Since $\pairing{\Gal(k_\infty)}{\zeta} = 0$, by shifting by $\zeta$ we may
    assume $\gamma \in H$.

    Replacing $k_0$ with $k_0(\gamma)$ if necessary, we may assume $\gamma \in
    \G(k_0)$.

    By \lemref{lem:subgroupsEndomorphisms} and the assumption that the
    endomorphisms are over $k_0$, we have that $H$ is defined over $k_0$.
    So since $H$ is divisible, $Z_\infty \leq  T^{H}_\infty$. It remains to see
    that the index is finite.

    Now $\G(k_\infty)$ is an $\er$-submodule of $\G(\Qbar)$ by
    the assumption that the endomorphisms are over $k_0 \leq  k_\infty$,
    and $\er\gamma$ is a finitely generated subgroup since $\er$ is finitely
    generated.
    So by \lemref{lem:boundedDiv}, say
    $m>0$ is such that $m\cdot\pureHull_{\G(k_\infty)}(\er\gamma) \leq  \er\gamma +
    \G[\infty]$.

    For $n>0$, let $Z_n := \pairing{\Gal(k_\infty)}{\gamma}_n \leq \G[n]$.
    Note that $Z_n$ is defined over $k_0$; indeed, if $\sigma \in \Gal(k_\infty)$ and
    $\tau \in \Gal(k_0)$, then
    $\sigma^{\tau^{-1}} \in \Gal(k_\infty)$ and
    \[ \pairing{\sigma^{\tau^{-1}}}{\gamma}_n = \tau\sigma\tau^{-1}\alpha
    - \alpha = \tau(\sigma(\tau^{-1}\alpha) - \tau^{-1}\alpha) =
    \tau\pairing{\sigma}{\gamma}_n \]
    (where $n\alpha = \gamma$, and hence $n\tau^{-1}\alpha=\gamma$).
    So $(Z_n)^\tau = Z_n$.

    So by \lemref{lem:faltings},
    the isogenous groups
    $B_n:=\quot{\G}{Z_n}$ fall into finitely many $k_0$-isomorphism classes.
    Therefore we may find $N$ such that for any $n$, there exists a $k_0$-isogeny
    $\theta_n : B_n \maps \G$ of degree $\deg\theta_n := \left|{\ker\theta_n}\right|$ dividing
    $N$.

    We conclude the proof of the Proposition by showing that for any $n$, the
    index $[H[n]:Z_n]$ divides $N \cdot \left|{\G[m]}\right|$.

    Indeed, let $\eta \in \er$ be the composition $\eta(x) :=
    \theta_n(\quot{x}{Z_n})$ of $\theta_n$ with the quotient map.
    Suppose $n\beta=\gamma$.
    Then $n\eta\beta=\eta\gamma$. But
    $\eta\beta$ is $\Gal(k_\infty)$-invariant;
    indeed, $Z_n \leq  \ker(\eta)$ and $\eta$ is defined over $k_0 \leq  k_\infty$,
    so \[ \sigma\eta\beta = \eta\sigma\beta =
    \eta(\beta + \pairing{\sigma}{\gamma}_n) = \eta\beta .\]
    So $\eta\beta \in \pureHull_{\G(k_\infty)}(\er\gamma)$,
    so $m\eta\beta \in \er\gamma + \G[\infty]$.
    So $m\eta\gamma \in n\er\gamma + \G[\infty]$,
    so $k(m\eta-n\eta')\gamma = 0$ for some $k>0$ and some $\eta' \in \er$.
    So by the choice of $H$, we have $m\eta = n\eta'$ on $H$.

    Hence $m\eta(H[n])=0$, i.e.\ $\theta_n(\quot{H[n]}{Z_n})\leq \G[m]$, and
    hence
    \begin{align*}
    [H[n]:{Z_n}]
    \quad&\divides\quad \left|{\ker\theta_n}\right| \cdot \left|{\G[m]}\right| \\
    &\divides\quad N \cdot \left|{\G[m]}\right|
    .\end{align*}
\end{proof}

\bibliography{mt,misc}
\end{document}